\documentclass{amsart}

\usepackage[utf8]{inputenc} 
\usepackage[T1]{fontenc}
\usepackage[english]{babel}
\usepackage{mathtools}
\usepackage{amssymb}
\usepackage{amsthm}
\usepackage{mathrsfs}
\usepackage{scalerel}

\usepackage{xcolor}
\usepackage{tikz,graphicx}

\usetikzlibrary{automata,positioning,arrows.meta,shapes.misc}
\usetikzlibrary{calc}

\usepackage{enumerate,url,float,lscape}
\usepackage{hyperref}

\usepackage[normalem]{ulem}

\newtheorem{lemma}{Lemma}[section]
\newtheorem{assumption}[lemma]{Assumption}
\newtheorem{notation}[lemma]{Notation}

\newtheorem{corollary}[lemma]{Corollary}
\newtheorem{definition}[lemma]{Definition}
\newtheorem{theorem}[lemma]{Theorem}
\newtheorem{remark}[lemma]{Remark}

\numberwithin{equation}{section}
\numberwithin{lemma}{section}

 \def\Graph{\mathcal{G}}

 \def\mV{\mathsf{V}}
 \def\mFB{\mathsf{FB}}
 \def\mIS{\mathsf{IS}}

 \def\mE{\mathsf{E}}

 \def\mK{\mathsf{K}}

 \def\mv{\mathsf{v}}
 \def\mvdi{\mv^*}
 \def\mVdi{\mV^*}
 \def\me{\mathsf{e}}
 
 \def\mw{\mathsf{w}}

 \def\mG{\mathsf{G}}

\DeclareMathOperator{\diag}{diag}

\newcommand{\C}{\mathbb{C}}
\newcommand{\N}{\mathbb{N}}

\newcommand{\ud}{\,\mathrm{d}}
\newcommand{\e}{\mathrm{e}}

\def\:{\thinspace:\thinspace}

\title
{Distinguishing co-spectral quantum graphs by scattering}

\author[D.~Mugnolo]{Delio Mugnolo}
\author[V.~Pivovarchik]{Vyacheslav Pivovarchik}

\address{Delio Mugnolo, Lehrgebiet Analysis, Fakultät Mathematik und Informatik, Fern\-Universität in Hagen, D-58084 Hagen, Germany}
\email{delio.mugnolo@fernuni-hagen.de}

\address{Vyacheslav Pivovarchik, South Ukrainian National Pedagogical University, 65920, Odesa, Ukraine}
\email{vpivovarchik@gmail.com}



\thanks{
This article is based upon work from COST Action 18232 MAT-DYN-NET, supported by COST (European Cooperation in Science and Technology), \url{www.cost.eu}.
DM was partially supported by the Deutsche Forschungsgemeinschaft (Grant 397230547).
}

\begin{document}
\begin{abstract}
We propose a simple method for resolution of co-spectrality of Schrödinger operators on metric graphs. Our approach consists of attaching a lead to them and comparing the $S$-functions of the corresponding scattering problems on these (non-compact) graphs. 

We show that in several cases -- including general graphs on at most 6 vertices, general trees on at most 9 vertices, and general fuzzy balls -- eigenvalues and scattering data are together sufficient to  distinguish co-spectral metric graphs.
\end{abstract}

\maketitle
	

\section{Introduction}
The classical inverse  (Schrödinger) problem of finding the potential on a half-axis using the scattering data has been in the spotlight of mathematicians since the 1950s. The $S$-function (traditionally called $S$-matrix) of the $S$-wave Schrödinger equation (Sturm--Liouville equation) alone cannot possibly determine the potential uniquely \cite {Bar}, but it does in combination with knowledge of the normal eigenvalues (if any) and the corresponding normalizing constants  \cite{Mar,Lev}. If, however, the potential is decreasing faster than any exponent at infinity, then the corresponding Jost function is an entire function (see, e.g., \cite[Theorem 7.2.2]{Nus}) and even knowledge of the $S$-function only is sufficient to find the potential \cite{Mar}. 

In the case of scattering on non-compact graphs, the situation is more complicated. 
Some results of recovering the potentials on the edges were obtained for particular classes of graphs in \cite{P1,LP,Ak}. Here we consider the problem of recovering the shape of the graph, rather than the potentials. It is well-known that even in the case of compact graph the shape of a metric graph $\Graph$ cannot generally be recovered from the spectrum of  (a suitable realization of) the free Laplacian $\Graph$, as co-spectral graphs may exist. Therefore, we consider a generalized problem that consists in recovering the shape of the graph from both the eigenvalues the Schrödinger operator \textit{and} some appropriate scattering data.

In discrete mathematics, the study of co-spectral (a.k.a.\ isospectral) graphs with respect to a specific choice of graph matrix is a classical topic in algebraic graph theory. The case of the adjacency matrix has been studied thoroughly since the 1950s. Other classical graph matrices include the Laplacian, normalized Laplacian, and signless Laplacian. Observe that the notions of co-spectrality with respect to all these matrices coincide if, in particular, the relevant graphs are regular (i.e., if all vertices have the same number of incident edges). In the 1980s, von Below developed a transference principle \cite{Bel85} that expresses the eigenvalues of the Laplacian on an equilateral metric graph $\Graph$ in terms of eigenvalues of the normalized Laplacian on its underlying combinatorial graph $\mG$ (as long as $\mG$ is simple, i.e., it contains neither loops nor multiple edges) along with the eigenvalues of the one-dimensional Laplacian with Neumann conditions. Combining this with the well-developed theory of co-spectral adjacency matrices, von Below could borrow the co-spectral pair of regular combinatorial graphs in~\cite[Figure~6.4]{CveDooSac79} and present a co-spectral pair of (equilateral) metric graphs~\cite[Figure~2]{Bel01}. 
 
More generally, von Below's above mentioned transference principle implies that if two simple graphs $\mG_1,\mG_2$ are co-spectral with respect to the normalized Laplacian (so, in particular they have the same number of vertices), and if additionally they have the same number of edges, then their equilateral metric counterparts $\Graph_1,\Graph_2$ are also co-spectral,
provided $\Graph_1,\Graph_2$ are both bipartite, or both non-bipartite. This immediately allows for lifting new co-spectrality results obtained in the combinatorial case to the metric case: we mention in particular several new constructions of classes of co-spectral graphs recently identified in \cite{ButGro11,LleFabPos22}.

While simplicity seems to be an important condition in order to apply von Below's transference principle, Roth  presented in \cite[Section V]{Rot84} a different (uncountable) class of non-simple co-spectral metric graphs: all of them are equilateral pumpkin chains, in the sense of \cite[Section~5]{BerKenKur19}, and in fact Roth's construction turns out to be a special case of \cite[Lemma~5.4]{BerKenKur19}.

In quantum graph theory it is usual to consider spectral problems associated with general Sturm--Liouville equations on equilateral metric graphs with standard conditions at the interior vertices and Neumann or Dirichlet conditions at the pendant vertices. Here, too, the problem of co-spectrality arises. Such investigations were started in \cite{KN,BKS} and continued in \cite{CP,CP2}.

In this paper we consider the problem of finding the shape of a compact graph with one \textit{lead} (i.e., a half-infinite edge) attached to it. 
 We always assume our metric graph to be simple and connected, consisting of an equilateral subgraph
and one additional lead. The potentials on the finite edges are real $L^2$-functions, whereas the potential on the lead is identically zero: this guarantees that the $S$-function is meromorphic. As already mentioned, we impose Neumann conditions at the pendant vertices and standard conditions at the interior vertices. In general, the knowledge of the $S$-matrix is not sufficient to determine the
topological structure of the graph uniquely: several negative results have been obtained in~\cite{KurSte02}, but
we show that 
knowledge of (asymptotics of) the $S$-matrix and the asymptotics of embedded eigenvalues uniquely determine the shape of the graph, if its compact part has at most 6 vertices: to this aim, we use the result of \cite{CP}, where the authors have proved that the spectrum of Schrödinger operator with standard vertex conditions on a simple, connected graph on at most 5 vertices
 uniquely determines the shape of the graph; as well as the analogous results of \cite{CP2,Pist} on graphs on 6 vertices.
(We recall that, in the case of graphs with non-commensurate edges, the spectrum always determines the shape of the graph \cite{GS}.)



 In Section 2 we formulate the spectral Sturm--Liouville problem on a simple, compact, connected equilateral graph endowed standard conditions at the interior vertices and Neumann conditions at the pendant vertices (Problem I) and another problem on the same graph with the same conditions at all vertices but one (the \textit{root}), where we impose the Dirichlet condition (Problem II). 

We can build upon known results from the literature -- specifically, \cite[Theorem 6.4.2]{MP2} -- that describe the relation between the spectrum of Problem I and the spectrum of the corresponding normalized Laplacian; as well as  the relation between the spectrum of Problem II and the spectrum of the modified normalized Laplacian corresponding to the so-called \textit{interior subgraph}, see Definition~\ref{def:matrices-d-a} below.


In Section 3 we consider a graph that consists of a lead attached to a compact, simple, connected equilateral subgraph. The potential on the lead is assumed to be zero identically. The corresponding scattering problem is a generalization of the so-called Regge problem \cite{Re}. We show that the essential (continuous) spectrum of the scattering problem (Problem III) covers the non-negative half-axis. There may occur eigenvalues embedded into continuous spectrum and normal eigenvalues (i.e.,  isolated eigenvalues of finite multiplicity) on the negative half-axis. In this section we describe connection between the $S$-function of Problem III and the characteristic functions of Problems I and II.

In Section 4 we show that the asymptotics of the $S$-function and the asymptotics of the eigenvalues uniquely determine the shape of the graph which consists of a lead attached to a simple connected equilateral subgraph if the number of vertices in the graph is at most 6.

In Section 5 we show that asymptotics of the $S$-function and of the eigenvalues uniquely determine the shape of a tree which consists of a lead attached to an equilateral subtree if the number of vertices in the tree does not exceed 9.

In Section 6 we consider a special class of \textit{fuzzy balls}: these represent an infinite class of co-spectral graphs and were first introduced by Butler and Grout in \cite[Section~4]{ButGro11}: we then show that attaching one leads either to the vertex of degrees $r$ or $s$ it is possible to distinguish non-isometric graphs.  Indeed, the class of co-spectral fuzzy balls identified in \cite[Example~1]{ButGro11} is much larger, and has been further extended in~\cite[Section~4]{LleFabPos22}: we leave for a later paper the extension of our results to this and further general configurations.

\section{Spectral problems on compact graphs}

Let $\Graph$ be a metric graph with vertex set $\mV$ and edge set $\mE$. We assume the edges and hence the vertices to be finitely many: we denote their number by $V$ and $E$, respectively  (to avoid trivialities, $E\ge 2$ and $V\ge 2$); and the generic vertex and edge by $\mv,\me$, respectively. Additionally, we assume $\Graph$ to be \textit{equilateral}, i.e., all edges to have same length $\ell\in (0,\infty)$: accordingly, there is a direct correspondence between $\Graph$ and the underlying combinatorial graph $\mG$. 
Given a vertex $\mv$, we denote by $\mE_{\mv}$ the set of edges incident in $\mv$ and by 
\begin{equation}\label{eq:degree-def}
\deg(\mv):=\#\mE_\mv
\end{equation}
 its \textit{degree}, i.e., the number of such edges. We denote by $\mV_{pen}$ the set of \textit{pendant vertices}, i.e.,
\[
\mV_{pen}:=\{\mv\in\mV:\deg(\mv)=1\}.
\]
(Let us remark that, unlike in many spectral geometric investigations we will not be free to insert or remove degree-2-vertices (``dummy vertices''), as $\Graph$ is assumed to be equilateral.)

Upon identifying each edge with (a copy of) the interval $[0,\ell]$, we will endow $\Graph$ with its canonical metric measure space structure (see~\cite{Mug19};
also, to avoid trivialities we assume $\Graph$ to be connected with respect to this metric.
Observe that the identification of each edge $\me$ with $[0,\ell]$ induces a natural orientation from $0$ (initial endpoint) to $\ell$ (terminal endpoint) on $\me$ and hence on $\Graph$ and we will denote by $\mE^-_\mv,\mE^+_\mv$ the sets of edges whose initial, respectively terminal endpoint is $\mv$.
\footnote{For our purposes  -- that is, the study of second order elliptic equations on each edge -- this orientation is essentially arbitrary and will not play any further role, once it is fixed.}

Finally, we always assume $\mG$ and hence $\Graph$ to be \textit{simple}, i.e., that for any two vertices $\mv,\mw\in\mV$ there is at most one edge simultaneously incident in both of them, and none if $\mv=\mw$.

Let us summarize our basic setting in the following.

\begin{assumption}\label{assum:graph}
$\Graph$ is a simple, finite, equilateral, connected metric graph.
\end{assumption}

Let us now set up our function analytical setting: we introduce the function spaces
\[
L^2(\Graph)\quad\hbox{and}\quad  C(\Graph)
\]
that  are canonically induced by the structure of $\Graph$ as metric measure space (we equivalently regard any $f\in L^2(\Graph)$ as a vector-valued function,
\[
f=(f_\me)_{\me\in\mE},\qquad \hbox{where}\qquad f_\me\in L^2(0,\ell)\hbox{ for any }\me\in\mE.)
\]
We consider functions $y$  on $\Graph$ of class $\bigoplus_{\me\in\mE}H^2(0,\ell)$: in our investigation, they are subject to the system of $E$ scalar Sturm--Liouville equations
\begin{equation}
-y_\me''(x)+q_\me(x)y_\me(x)=\lambda y_\me(x), \qquad \me\in\mE,\ x\in (0,\ell),
\label{2.1}
\end{equation}
 for some real-valued $q_{\me}\in L^2(0,\ell)$.

In order to turn~\eqref{2.1} into a self-adjoint problem, we have to impose appropriate transmission conditions on each $y\in \bigoplus_{\me\in\mE}H^2(0,\ell)$ at the vertices of $\Graph$: unless otherwise stated, we impose \textit{standard transmission conditions} consisting of 
\begin{itemize}
\item \textit{continuity conditions}: $y\in C(\Graph)$
\footnote{This, in particular, allows for the introduction of the point evaluation at the vertices of $\Graph$, hence the notation $y(\mv)$ is justified.}; and

\item \textit{Kirchhoff condition}: 
 \begin{equation}
\sum_{\me\in \mE^+_\mv} y'_{\me} (\ell)=\sum_{\me\in \mE^-_\mv} y'_{\me} (0),\qquad \mv\in\mV,
 \label{2.3}
 \end{equation}
 with respect to the (arbitrary but fixed) orientation of $\mG$ introduced above.
\end{itemize}

We will in the following apply some ideas from scattering theory to our context: to this purposes, for any given distinguished vertex $\mvdi$ we consider the usual Sturm--Liouville-type problem associated with different transmission conditions at $\mvdi$:

\begin{itemize}
\item Problem I (the \textit{Neumann problem}) consists of \eqref{2.1} endowed with standard conditions at all vertices of $\Graph$, including $\mvdi$;
\item Problem II (the \textit{Dirichlet problem}) consists of \eqref{2.1} endowed with standard conditions at all vertices of $\Graph$ apart from those in $\mVdi\subset \mV$, on which Dirichlet conditions
\[
y(\mv)=0,\qquad \mv\in\mVdi,
\]
are imposed.
\end{itemize}
(At the risk of being pedantic, let us stress that the Problem II is strictly more general, as it includes Problem I as a special case for $\mVdi=\emptyset$.)

We will in the following for all $\lambda\in \C$ denote by 
\begin{equation}\label{eq:sc}
(0,\ell)\ni x\mapsto s(\sqrt{\lambda},x )\in \C
\qquad\hbox{and}\qquad
(0,\ell)\ni x\mapsto c(\sqrt{\lambda},x)\in \C
\end{equation}
the functions whose component $s_\me,c_\me$, $\me\in\mE$, is on each edge the unique solution of the Sturm--Liouville equation \eqref{2.1} that
satisfies the conditions
\begin{equation}\label{eq:se}
s_\me(\sqrt{\lambda},0)=s_\me'(\sqrt{\lambda},0)-1=0
\end{equation}
and
\begin{equation}\label{eq:ce}
c_\me(\sqrt{\lambda},0)-1=c_\me'(\sqrt{\lambda},0)=0,
\end{equation}
respectively.

\begin{remark}
We stress that $s_\me,c_\me$ do generally depend on $\me$, since so do the potentials $q_\me$. However, when we consider the case of zero potentials -- i.e., $q_\me\equiv 0$ for all $\me\in\mE$ -- which in particular fulfill  Assumption 2.5, then we are going to drop the indices and simply write $s(\lambda,x)=\hat{s}(\lambda,x)=\frac{\sin\sqrt{\lambda} x}{\sqrt{\lambda}}$ and $c(\lambda,x)=\hat{c}(\lambda,x)=\cos(\sqrt{\lambda}x)$.
\end{remark}

In order to find a characteristic function of our Sturm--Liouville problems, following~\cite{Pokor} we look for coefficients $A_{\me_1},\ldots,A_{\me_{\me}},B_{\me_1},\ldots,B_{\me_{\me}}$ such that the the solution of \eqref{2.1} can be expressed in the form 
$$
Y=
\begin{pmatrix}
y_{\me_1}\\
\vdots\\
y_{\me_{\mE}}
\end{pmatrix}
 \qquad\hbox{for }\qquad  y_{\me}(x)=A_{\me}s_{\me}(\lambda,x)+B_{\me}c_{\me}(\lambda,x),\qquad \me\in\mE,\ x\in (0,\ell).
$$
Substituting this into the continuity conditions; as well as into Kirchhoff condition at each interior vertex and into the Neumann conditions at all pendant vertices for the Problem I (resp., Dirichlet conditions and $\mV^*$ and Neumann conditions at $\mV_{pen}\setminus\mV^*$  for the Problem II), 
we obtain a system of $2E$ linear algebraic equations with unknowns $A_{\me_1},\ldots, A_{\me_E},B_{\me_1},\ldots,B_{\me_E}$.
Denote the $2E\times 2E$ matrix of this system by $\|\Phi_N(\lambda)\|$ (resp., $\|\Phi_D(\lambda)\|$): we call it the \textit{characteristic matrix} of our problem. Observe that it involves the values $s_{\me}(\lambda,\ell)$, $s'_{\me}(\lambda,\ell)$, $c_{\me}(\lambda,\ell)$, $c'_{\me}(\lambda,\ell)$. 
Then the equation
$$
\det\|\Phi_N(\lambda)\|=0\qquad\hbox{(resp., $\det\|\Phi_N(\lambda)\|=0$)}
$$
fully determines the spectrum of~\eqref{2.1} under relevant vertex conditions.

\begin{definition}
We call
 \begin{equation}
 \label{2.9}
\lambda\mapsto \phi_N(\lambda):=\det(\Phi_N(\lambda))\qquad\hbox{(resp., $\lambda\mapsto\phi_D(\lambda):=\det\|\Phi_D(\lambda)\|$)}
 \end{equation}
 the \emph{characteristic
function} of Problem I (resp., of Problem II).
\end{definition}
This characteristic function is determined up to a constant multiple and its set of zeros coincides with the spectrum of Problem I, i.e.\ the Neumann problem (resp., of Problem II, i.e.\ the Dirichlet problem)

 
Let us now introduce another notion that will be important in the following.

\begin{definition}\label{def:matrices-d-a}
Let $\Graph$ satisfy Assumption~\ref{assum:graph} and let $\mVdi\subset\mV$. We call \emph{interior subgraph of $\Graph$} (with respect to $\mVdi$) the graph $\hat{\Graph}$ obtained by deleting from $\Graph$  the vertex set $\mVdi$ along with all edges of $\Graph$ that are incident in $\mVdi$; observe that the vertex set of $\hat{\Graph}$ is $\mV\setminus\mVdi$. 

Let $\hat{A}$ be the adjacency matrix of $\hat{\Graph}$ and let $\hat{D}_\Graph$ be the matrix
\footnote{We reiterate that, in accordance with \eqref{eq:degree-def}, $\deg(\mv)$ is the degree of the vertex $\mv$ in $\Graph$ -- not in $\hat{\Graph}$!}
\[
\hat{D}_\Graph :=\diag(\deg(\mv))_{\mv\notin\mVdi}.
\]

Then both $\hat{A}$ and $\hat{D}_\Graph$ are $\#(\mV\setminus\mVdi)\times \#(\mV\setminus\mVdi)$-matrices and we call
 $$
P_{\Graph,\hat{\Graph}}(z):=\det(z\hat{D}_\Graph -\hat{A}), \qquad z\in \C.
 $$
 the \emph{characteristic polynomial} of $\hat{\Graph}$.
 \end{definition}
 
\begin{remark}
If $\mVdi=\emptyset$ and thus $\hat{\Graph}=\Graph$, then the characteristic polynomial $P_{\Graph,\Graph}(\cdot)=\det(\cdot D_{\Graph}-A)$ is the determinant of the discrete Laplacian of the underlying combinatorial graph $\mG$.
\end{remark}

 Let us now formulate two sets of assumptions on the potentials that we will need to impose in our main results. 
\begin{assumption}\label{assum:potential-weak}
For all $\me\in\mE$ the potential $q_\me$ is a real-valued function of class $L^2(0,\ell)$.
\end{assumption}
\begin{assumption}\label{assum:potential}
The potential $q$ is edge-independent (i.e., there exists $q\in L^2(0,\ell)$ such that $q_\me\equiv q$) and symmetric with respect to the midpoint of any edge (i.e., $q(\ell-x)= q(x))$ for a.e.\ $x\in (0,\ell)$). 
\end{assumption} 


 The following was proved in \cite[Theorem 6.4.2]{MP2}.

\begin{theorem}\label{thm:main}
Under the  Assumptions \ref{assum:graph} and \ref{assum:potential} the spectrum of  Problem II
 coincides with the set of zeros of the characteristic function 
\begin{equation}
\label{2.11-dir}
\C\ni\lambda\mapsto\phi_{D}(\lambda):=s^{E-V+r}(\sqrt{\lambda},\ell)P_{\Graph,\hat{\Graph}}(c(\sqrt{\lambda},\ell))\in \C,
\end{equation}
where $s(\sqrt{\lambda},\cdot)$ and $c(\sqrt{\lambda},\cdot)$ are the solutions of the Sturm--Liouville equation for initial conditions $s(\sqrt{\lambda}, 0)=s'(\sqrt{\lambda}, 0)-1=0$ and $c(\sqrt{\lambda},0)-1=c'(\sqrt{\lambda},0)=0$ and $r=\# \mV^*$. 
\end{theorem}

\begin{notation}
If $\mVdi=\emptyset$, we introduce the characteristic function $\phi_N$ defined by
\begin{equation}
\label{2.11}
\C\ni\lambda\mapsto \phi_{N}(\lambda):=s^{E-V}(\sqrt{\lambda},\ell)P_{\Graph,\Graph}(c(\sqrt{\lambda},\ell))\in\C.
\end{equation}
\end{notation}



 

\section{Attaching leads to compact metric graphs}\label{sec:scattering}

In this section we consider a graph $\Graph_\infty$ obtained by attaching a lead $\me_0$ to $\mvdi$ which is the root of a compact simple connected equilateral graph $\Graph$ described in the previous section. We direct this edge $\me_0$ away from $\mvdi$. We assume that the potential $q_0$ on $\me_0$ is identically $0$. Thus we have the equations
\begin{equation}
\label{3.1}
-y_\me^{\prime\prime}(x)+q_\me(x)y_\me(x)=\lambda y_\me(x),\qquad  \me\in\mE ,\ x\in [0,\ell],
\end{equation}
on all finite edges along with the equation
\begin{equation}
\label{3.2}
-y_0^{\prime\prime}(x)=\lambda y_0(x),\qquad x\in [0,\ell],
\end{equation}
on the edge $\me_0$.
We are going to endow the Sturm--Liouville equation~\eqref{3.1}--\eqref{3.2} on $\Graph_\infty$ with standard conditions at all vertices.

%
%

For a closed linear operator $A$ on a Hilbert space, we let $D(A)$, $\rho(A)$
and $\sigma(A)$ denote its domain, resolvent set and spectrum. We
refer to \cite[Section I.2]{GK} for the definition of \textit{normal} (that is,
isolated eigenvalues of finite multiplicity) eigenvalues, and denote by $\sigma_{0}(A)$ the
set of normal eigenvalues of $A$ and by
$\sigma_{\mathrm{ess}}(A)=\sigma(A)\backslash\sigma_{0}(A)$ the essential
spectrum. 
At this
point we recall that the spectrum of any self-adjoint operator $A$
coincides with its approximate spectrum, see, e.g.,
\cite[page 118]{EE}, where the latter is defined as the set of
$\lambda\in\mathbb{C}$ such that there exists a sequence
$\left\{f_{n}\right\}_{n=1}^\infty$ in $D(A)$
such that $\| f_{n}\|\equiv 1$ and $(\lambda I-A)f_{n}\to 0$ as $n\to\infty$. If
the sequence $\left\{f_{n}\right\}_{n=1}^{\infty}$ is compact,
then $\lambda$ is either a normal eigenvalue, or an eigenvalue
that belongs to the essential spectrum (in the latter case, in
quantum mechanics, $\lambda$ is called a \textit{bound state embedded
into the continuous spectrum}).

On the Hilbert space 
\[
L^2(\Graph_\infty):=L^2(0,\infty)\oplus \bigoplus\limits_{j=1}^E L_{2}(0,\ell)
\]
of square-integrable vector-valued functions $y=(y_\me)_{j=0}^E$ we
introduce an operator $A$, related to the boundary value problem
\eqref{3.1}--\eqref{3.2}, that acts as
$$
A(y_\me)_{\me\in\mE}=(-y^{\prime\prime}_{j}+q_{j}y_{j})_{\me\in\mE},
$$ 
(we recall that $q_0(x)\equiv 0$) with domain
 \begin{equation} 
 \label{3.8}
D(A):=\left\{y=(y_\me)_{j=0}^E \in C(\Graph_\infty)\cap L^2(\Graph_\infty):y''\in L^2(\Graph_\infty)\hbox{ and }\sum_{\me\in\mE^-_\mv}y'(0)=\sum_{\me\in\mE^+_\mv}y'(\ell)\hbox{ for all }\mv\in\mV \right\}.
\end{equation}

We identify the spectrum of the operator $A$ with the spectrum of the boundary problem \eqref{3.1}--\eqref{3.2}. The followings seems to be folklore, but we could not find an appropriate reference in the literature  and therefore we give a proof.

\begin{theorem}
Under the Assumptions~\ref{assum:graph} and \ref{assum:potential-weak}, let $q_0=0$.
 Then the 
the operator $A$ on $L^2(\Graph_\infty)$ is self-adjoint and
bounded from below.
Furthermore, $\sigma_{\mathrm{ess}}(A)=[0,\infty)$.
\end{theorem}

We refer to the spectrum of the operator $A$ as the \textit{spectrum of Problem III}.

\begin{proof}
 The proof is similar to the proof of \cite[Theorem 2.2]{LP}.
First, we claim that $A$ is symmetric. Indeed, for
$y=(y_\me)_{j=0}^E\in D(A)$ and $ z=(z_\me)_{j=0}^E\in D(A)$,
integrating by parts and using the boundary conditions described in the domain $D(A)$, we obtain:

\begin{equation}
\label{3.11}
\begin{split}
(Ay,z)&=\sum_{\me\in\mE}\int_0^\ell y_\me^{\prime\prime}\overline{z_\me}\ud x-
\int_0^{\infty}y_0^{\prime\prime}\overline{z_0}\ud x+
\sum_{\me\in\mE}\int_0^\ell q_\me y_\me\overline{z_\me}\ud x\\
&=
-\sum_{\me\in\mE}y_\me'(\ell)\overline{z_\me(\ell)}
+y_0'(0)\overline{z_0(0)}+
\sum_{\me\in\mE}\int_0^\ell y_\me'\overline{z_\me}'\ud x+
\int_0^{\infty}y_0'\overline{z_0}'\ud x+
\sum_{\me\in\mE}\int_0^\ell q_\me y_\me\overline{z_\me}\ud x.
\end{split}
\end{equation}

Since $y,z\in
D(A)$, $\overline{z}$ is continuous at all vertices and $y$ satisfies Kirchhoff conditions: therefore
\begin{equation}
(Ay,z)= \sum_{\me\in\mE}\int_0^\ell y_\me'\overline{z_\me}'\ud x+
\int_0^{\infty}y_0'\overline{z_0}'\ud x
+ \sum_{\me\in\mE}\int_0^\ell q_\me y_\me\overline{z_\me}\ud x=(y,Az), \label{3.12}
\end{equation}
thus proving the claim. Letting $z:=y$ in \eqref{3.11}, we obtain 
\begin{equation}
\label{3.13}
(Ay,y)=\sum_{\me\in\mE}\int_0^\ell |y_\me'|^2\ud x+
\int_0^{\infty}|y_0'|^2\ud x
+ \sum_{\me\in\mE}\int_0^\ell q_\me|y_\me|^2\ud x.
\end{equation}
Using the description of the domain of $A^*$, as in \cite[Section~7.5]{Rich}, it follows that $A$ is self-adjoint.
The operator $A$ is a self-adjoint extension of the operator $A_0$ defined by 
\[
A_0(y_\me)_{\me\in\mE}=(-y_\me^{\prime\prime}+q_\me y_\me)_{\me\in\mE}
\]
 with domain
\begin{equation}\label{3.14}
D(A_0):=H^2_0(0,\infty)\oplus \bigoplus_{j=1}^E H^2_0(0,\ell).
\end{equation}
The operator $A_0$ is the direct sum of symmetric, closed operators that are bounded from below, see, e.g., \cite[Theorem V.19.5]{Nai}. Therefore, $A_0$ is also symmetric, closed and bounded from below.
Furthermore, using \cite[Theorem 1 /31(3)/]{Gla}, we conclude that the part of the spectrum of $A$ located below $0$ consists of at most finitely many normal eigenvalues.
\end{proof}

Now, denote
\[
\hat{s}(\lambda, x):=\frac{\sin\sqrt{\lambda}x}{\sqrt{\lambda}}\quad\hbox{and}\quad 
\hat{c}(\lambda,x):=\cos\sqrt{\lambda}x,\qquad \lambda\in\C.
\]
Arguments similar to those used in proof of \cite[Theorem 2.1]{LawP} show that the 
restriction
 of the solution of problem \eqref{3.1}--\eqref{3.2} onto the edge $\me_0$ is
$$
y_0(\lambda,x)=\phi_N(\lambda)\hat{s}(\lambda,x)+\phi_D(\lambda)\hat{c}(\lambda,x)
$$
(or, correspondingly, 
\begin{equation}
\label{3.15}
y_0(\lambda,x)=\frac{1}{2i\sqrt{\lambda}}
\left(\e^{i\sqrt{\lambda}x}(\phi_N(\lambda)+i\sqrt{\lambda}\phi_D(\lambda))
-\e^{-i\sqrt{\lambda}x}(\phi_N(\lambda)-i\sqrt{\lambda}\phi_D(\lambda))\right)
\end{equation} 
if $\mvdi$ is a pendant vertex.)
By analogy with the classical $S$-function (see, e.g., \cite{Mar}) we introduce the meromorphic function
$$
S:\lambda\mapsto \frac{E(\sqrt{\lambda})}{E(-\sqrt{\lambda})},
$$
where
\begin{equation}
\label{3.16}
E(\lambda):=\phi_N(\lambda)+i\sqrt{\lambda}\phi_D(\lambda),\qquad \lambda\in\C,
\end{equation}
is the Jost function.

\begin{lemma}\label{lem:interl}
Under the Assumption~\ref{assum:graph} and \ref{assum:potential-weak}, let $q_0=0$, and let $\mV^*$ be a singleton. Then the Jost function $E(\sqrt{\lambda})$ is an entire function of $\sqrt{\lambda}$ with the zeros 
lying in the closed lower half-plane $\overline{\mathbb{C}^-}$ symmetrically with respect to the imaginary axis and on a finite interval of the negative imaginary half-axis.
\end{lemma}

\begin{proof}
The assumption $\#\mV^*=1$ means that one lead is  attached to a pendant vertex. Then
using \cite[Theorem 3.1.8]{BK} we conclude that the zeros $\{\mu_k\}_{k=1}^{\infty}$ of $\phi_N$ interlace with the zeros $\{\nu_k\}_{k=1}^{\infty}$ of $\phi_D$:
$$
\mu_1\leq\nu_1\leq\mu_2\leq\nu_2\leq \ldots\leq\mu_k\leq\nu_k\leq\ldots
$$

This means that the function $\frac{\phi_D}{\phi_N}$ is an essentially positive Nevanlinna function ($\in N^{ep}_+$, see \cite[page 456]{PW} or \cite[Definition 5.1.26]{MP}). Using \cite[Theorem 5.2.9]{MP} we obtain the statement. 
\end{proof}

\begin{remark}
A proof similar to that of Lemma~\ref{lem:interl} shows that the assertion carries over to the case of larger set $\mV^*$, too.
\end{remark}

 The eigenvalues embedded into the continuous spectrum are located on the half-axis $[0, \infty)$. They correspond to those $\lambda>0$ for which $\phi_N(\lambda)=\phi_D(\lambda)=0$. 
 The eigenfunctions corresponding to such eigenvalues are supported on the compact part of the graph. If $\phi_N(\lambda)=\phi_D(\lambda)=0$ for $\lambda<0$, then this $\lambda$ is a normal eigenvalue.

\begin{lemma}\label{lem:asympt}
Under the Assumption~\ref{assum:graph} and \ref{assum:potential-weak}, let $q_0=0$.
Then the following asymptotics hold:
\begin{equation}
\label{3.17}
\phi_N(\lambda)\stackrel{\lambda\to +\infty}{=}\hat{\phi}_N(\lambda)+o(1), \
\phi_D(\lambda)\stackrel{\lambda\to +\infty}{=}\hat{\phi}_D(\lambda)+o(1)
\end{equation}
where `hat' in $\hat{\phi}_N$ and $\hat{\phi}_D$ corresponds to the case of identically zero potentials on all edges.
\end{lemma}

 We will in the following refer to $\hat{\phi}_N,\hat{\phi}_D$ as the \emph{leading term} of the characteristic function $\phi_N,\phi_D$, respectively.

\begin{proof}
It is known (see, e.g., \cite[Lemma 3.4.2]{Mar}) that 
$$
s_\me(\lambda,\ell)\stackrel{|\lambda|\to \infty}{=}\frac{\sin\sqrt{\lambda}\ell}{\sqrt{\lambda}} +O\left(\frac{\e^{|{\rm Im}\sqrt{\lambda}\ell|}}{|\lambda|}\right), \quad
c_\me(\lambda,\ell)\stackrel{|\lambda|\to \infty}{=}\cos\sqrt{\lambda}\ell +O\left(\frac{\e^{|{\rm Im}\sqrt{\lambda}\ell|}}{|\sqrt{\lambda}|}\right),
$$
Substituting these expressions into \eqref{2.9} we arrive at the statements of Lemma~\ref{lem:asympt}.
\end{proof} 

Substituting \eqref{3.17} into \eqref{3.16} we immediately obtain the following.

\begin{corollary}\label{cor:asympt}
Under the Assumption~\ref{assum:graph} and \ref{assum:potential-weak}, let $q_0=0$. Then the asymptotics
\begin{equation}
\label{3.18}
S({\lambda})\stackrel{{\lambda\to +\infty}}{=}\hat{S}(\sqrt{\lambda})(1+o(1)),
\end{equation}
and
\begin{equation}
\label{3.19}
E(\sqrt{\lambda})\stackrel{{\lambda\to +\infty}}{=}\hat{E}(\sqrt{\lambda})(1+o(1)).
\end{equation}
hold, where quantities with `hat' correspond to the case of identically zero potentials on all edges.
\end{corollary}

\section{Inverse problem for general small graph}

The function $S$ can be found from scattering experiments, then using \eqref{3.18} we find $\hat{S}$.
In this section we consider the problem of recovering the shape of a graph using the scattering data. By scattering data we mean the $S$-function along with the eigenvalues. In fact, in Theorem~\ref{thm:determ-6v} we show how to recover the shape of a graph as soon as we are given the leading term $\hat{E}$ of the asymptotics of the Jost function $E$.

Before proceeding with solving of the inverse problem where $\hat{E}(\sqrt{\lambda})$ is the given data, we make the following remark. In general, the Jost function $E$ is not uniquely determined by the scattering function $S$, as long as $E$ is allowed to have zeros on the real axis or pairs of pure imaginary zeros symmetric with respect to the origin. To illustrate this, let us suppose that $\sqrt{\lambda}_k$ is a real zero of $E(\sqrt{\lambda})$. Then, due to the symmetry (Lemma 3.2), $-\sqrt{\lambda_k}$ is also a zero of $E(\sqrt{\lambda})$, and $\pm\sqrt{\lambda}_k$ are zeros of $E(-\sqrt{\lambda})$, too. Cancellation of the correspondent factors in the fraction $
S(\lambda)=\frac{E(\sqrt{\lambda})}{E(-\sqrt{\lambda})}
$
shows that the scattering function $S(\sqrt{\lambda})$ does not change as long as we move zeros of $E(\sqrt{\lambda})$ along the real axis in symmetric fashion. Similarly we can achieve the same cancellation if we suppose that $E(\sqrt{\lambda})$ has two symmetrically located pure imaginary zeros $\sqrt{\lambda}_k=i|\sqrt{\lambda}_k|$ and $\sqrt{\lambda}_{-k}=-i|\sqrt{\lambda}_k|$ : indeed, in this case we can move $\sqrt{\lambda_k}$ and $\sqrt{\lambda_{-k}}$ along the imaginary axis preserving symmetry $|\sqrt{\lambda_k}|=|\sqrt{\lambda_{-k}}|$ and having $S(\sqrt{\lambda})$ unchanged. However, if we exclude these possibilities, that is, if we assume a priori that $E (\sqrt{\lambda})$ may have only one real zero which is simple and located at the origin and does not have any pairs of purely imaginary zeros that are symmetric about the origin, then the Jost function is uniquely determined by the scattering function. Indeed, the $S$-function is meromorphic due to $q_0(x)\equiv 0$, on the lead,
hence the zeros of $S$ are the zeros of $E$.

If there are eigenvalues then the corresponding factors in the numerator and denominator are cancelled. However, in this case knowing the $S$-function and the location of all eigenvalues we can multiply the numerator of $S(\sqrt{\lambda})$ by $\mathop{\prod}\limits_i(\sqrt{\lambda}-\sqrt{\alpha_i})$ where $\{\sqrt{\alpha_i}\}$ is the set (finite or infinite) of real and pure imaginary zeros of 
$E(\sqrt{\lambda})$ and obtain $E(\sqrt{\lambda})$. Using \eqref{3.15} we find $\hat{E}(\sqrt{\lambda})$ and
$$
\hat{\phi}_N(\lambda)=\frac{\hat{E}(\sqrt{\lambda})+\hat{E}(-\sqrt{\lambda})}{2},\qquad  \hat{\phi}	_D(\lambda)=\frac{\hat{E}(\sqrt{\lambda})-\hat{E}(-\sqrt{\lambda})}{2i\sqrt{\lambda}}.
$$
 
Does the Jost function $E$ uniquely determine the shape of the graph? Suppose there are two graphs with the same Jost function: according to \eqref{3.16}, the graphs must have the same 
 $\phi_N$ and the same $\phi_D$. 
It is known (see \cite{CP}) that there is no pairs of non-isometric graphs with the same (up to an arbitrary constant factor) $\hat{\phi}_N$ among the simple equilateral connected graphs with no more than 5 vertices.

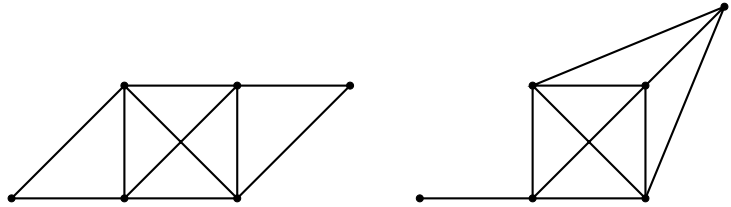
\begin{figure}[H]
\centering
\begin{tikzpicture}[scale=0.75]
\coordinate (a1) at (0,0);
\coordinate (a2) at (0,2);
\coordinate (a3) at (2,0);
\coordinate (a4) at (2,2);
\coordinate (a5) at (-2,0);
\coordinate (a6) at (4,2);
\draw[thick] (a1) -- (a2) -- (a3) -- (a4) -- (a1) -- (a5) -- (a2);
\draw[thick] (a1) -- (a3) -- (a6) -- (a2);
\draw[fill] (a1) circle (1.75pt);
\draw[fill] (a2) circle (1.75pt);
\draw[fill] (a3) circle (1.75pt);
\draw[fill] (a4) circle (1.75pt);
\draw[fill] (a5) circle (1.75pt);
\draw[fill] (a6) circle (1.75pt);
\end{tikzpicture}
\qquad
\begin{tikzpicture}[scale=0.75]
\coordinate (a1) at (0,0);
\coordinate (a2) at (0,2);
\coordinate (a3) at (2,0);
\coordinate (a4) at (2,2);
\coordinate (a5) at (-2,0);
\coordinate (a6) at (3.4,3.4);
\draw[thick] (a1) -- (a2) -- (a3) -- (a4) -- (a1) -- (a5);
\draw[thick] (a1) -- (a3) -- (a6) -- (a2) -- (a4) -- (a6);
\draw[fill] (a1) circle (1.75pt);
\draw[fill] (a2) circle (1.75pt);
\draw[fill] (a3) circle (1.75pt);
\draw[fill] (a4) circle (1.75pt);
\draw[fill] (a5) circle (1.75pt);
\draw[fill] (a6) circle (1.75pt);
\end{tikzpicture}\\[5pt]
\caption{Non-isomorphic graphs of 6 vertices with the same first and second terms of the asymptotics of the Sturm-Liouville problem eigenvalues.}\label{fig.2}
\end{figure}

%

Among the simple equilateral connected graphs of 6 vertices there is only one pair whose discrete Laplacians have same determinant, see \cite{Pist}. These graphs are shown in Figure~\ref{fig.2}; observe that they are both fuzzy balls in the sense of \cite{ButGro11}, see also Section~\ref{sec:exa-butgrout}. The determinant of the discrete Laplacian of these graphs is 
$$
P_{\Graph,\Graph}(z)=(z-1)(4z+1)^3(4z^2+z-1),
$$
see \cite{CP2}. Since identically zero potentials on the edges satisfy the Assumption~\ref{assum:potential}, then by Theorem~\ref{thm:main} we obtain  the expression
\begin{equation}
\label{4.1}
\hat{\phi}_N(\lambda)=\left(\hat{s}(\sqrt{\lambda},\ell)\right)^4(\hat{c}(\sqrt{\lambda},\ell)-1)(4\hat{c}(\sqrt{\lambda},\ell)+1)^3(4\hat{c}^2(\sqrt{\lambda},\ell)+\hat{c}(\sqrt{\lambda},\ell)-1).
\end{equation}
for the characteristic function of Problem I for the free Laplacian.
We need to compare also the corresponding functions $\phi_D$, but the form of this function depends on the vertex to which a lead is attached. All non-isometric graphs obtained by attaching a lead to  the left graph in Figure~\ref{fig.2} are shown in Figure~\ref{fig.3}.

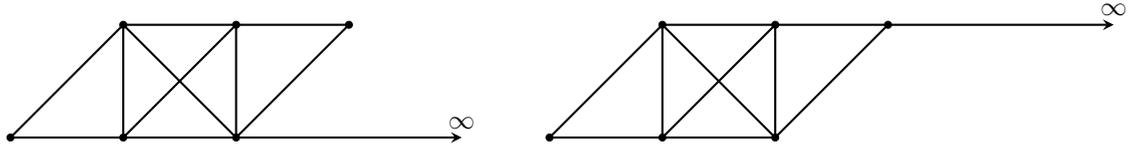
\begin{figure}[H]
\centering
\begin{tikzpicture}[scale=0.75]
\coordinate (a1) at (0,0);
\coordinate (a2) at (0,2);
\coordinate (a3) at (2,0);
\coordinate (a4) at (2,2);
\coordinate (a5) at (-2,0);
\coordinate (a6) at (4,2);
\coordinate (infty) at (6,0);
\draw[thick] (a1) -- (a2) -- (a3) -- (a4) -- (a1) -- (a5) -- (a2);
\draw[thick] (a1) -- (a3) -- (a6) -- (a2);
\draw[thick,-stealth] (a3) -- (infty);
\draw[fill] (a1) circle (1.75pt);
\draw[fill] (a2) circle (1.75pt);
\draw[fill] (a3) circle (1.75pt);
\draw[fill] (a4) circle (1.75pt);
\draw[fill] (a5) circle (1.75pt);
\draw[fill] (a6) circle (1.75pt);
\node at (infty) [anchor=south] {$\infty$};
\end{tikzpicture}
\qquad
\begin{tikzpicture}[scale=0.75]
\coordinate (a1) at (0,0);
\coordinate (a2) at (0,2);
\coordinate (a3) at (2,0);
\coordinate (a4) at (2,2);
\coordinate (a5) at (-2,0);
\coordinate (a6) at (4,2);
\coordinate (infty) at (8,2);
\draw[thick] (a1) -- (a2) -- (a3) -- (a4) -- (a1) -- (a5) -- (a2);
\draw[thick] (a1) -- (a3) -- (a6) -- (a2);
\draw[thick,-stealth] (a6) -- (infty);
\draw[fill] (a1) circle (1.75pt);
\draw[fill] (a2) circle (1.75pt);
\draw[fill] (a3) circle (1.75pt);
\draw[fill] (a4) circle (1.75pt);
\draw[fill] (a5) circle (1.75pt);
\draw[fill] (a6) circle (1.75pt);
\node at (infty) [anchor=south] {$\infty$};
\end{tikzpicture}
\\[5pt]
\caption{All non-isometric graphs obtained by attaching a lead to the left graph in Figure~\ref{fig.2}.}\label{fig.3}
\end{figure}

 Likewise, all non-isometric graphs obtained by attaching a lead to different vertices of the right graph in Figure~\ref{fig.2} are shown in Figure~\ref{fig.4}.
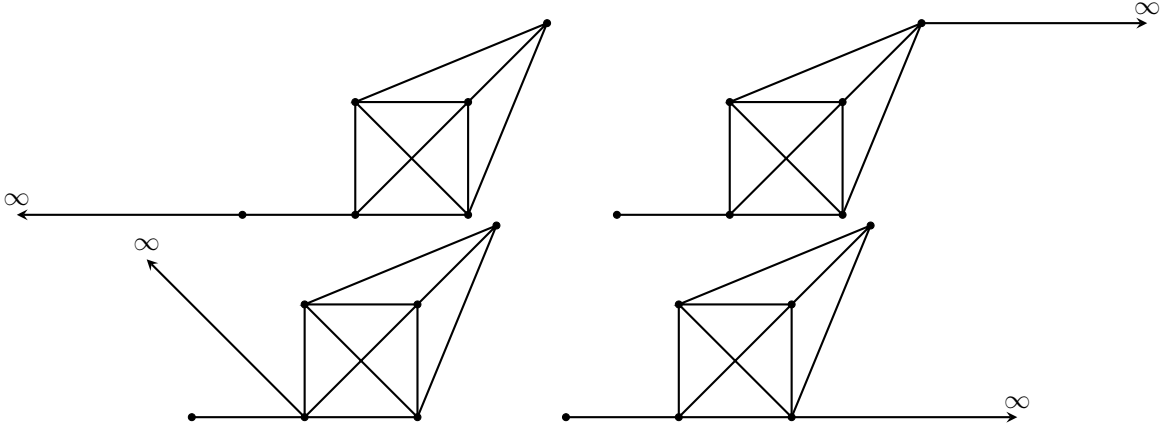
\begin{figure}[H]
\centering
\begin{tikzpicture}[scale=0.75]
\coordinate (a1) at (0,0);
\coordinate (a2) at (0,2);
\coordinate (a3) at (2,0);
\coordinate (a4) at (2,2);
\coordinate (a5) at (-2,0);
\coordinate (a6) at (3.4,3.4);
\coordinate (infty) at (-6,0);
\draw[thick] (a1) -- (a2) -- (a3) -- (a4) -- (a1) -- (a5);
\draw[thick] (a1) -- (a3) -- (a6) -- (a2) -- (a4) -- (a6);
\draw[thick,-stealth] (a5) -- (infty);
\draw[fill] (a1) circle (1.75pt);
\draw[fill] (a2) circle (1.75pt);
\draw[fill] (a3) circle (1.75pt);
\draw[fill] (a4) circle (1.75pt);
\draw[fill] (a5) circle (1.75pt);
\draw[fill] (a6) circle (1.75pt);
\node at (infty) [anchor=south] {$\infty$};
\end{tikzpicture}
\qquad
\begin{tikzpicture}[scale=0.75]
\coordinate (a1) at (0,0);
\coordinate (a2) at (0,2);
\coordinate (a3) at (2,0);
\coordinate (a4) at (2,2);
\coordinate (a5) at (-2,0);
\coordinate (a6) at (3.4,3.4);
\coordinate (infty) at (7.4,3.4);
\draw[thick] (a1) -- (a2) -- (a3) -- (a4) -- (a1) -- (a5);
\draw[thick] (a1) -- (a3) -- (a6) -- (a2) -- (a4) -- (a6);
\draw[thick,-stealth] (a6) -- (infty);
\draw[fill] (a1) circle (1.75pt);
\draw[fill] (a2) circle (1.75pt);
\draw[fill] (a3) circle (1.75pt);
\draw[fill] (a4) circle (1.75pt);
\draw[fill] (a5) circle (1.75pt);
\draw[fill] (a6) circle (1.75pt);
\node at (infty) [anchor=south] {$\infty$};
\end{tikzpicture}
\qquad
\begin{tikzpicture}[scale=0.75]
\coordinate (a1) at (0,0);
\coordinate (a2) at (0,2);
\coordinate (a3) at (2,0);
\coordinate (a4) at (2,2);
\coordinate (a5) at (-2,0);
\coordinate (a6) at (3.4,3.4);
\coordinate (infty) at (-2.8,2.8);
\draw[thick] (a1) -- (a2) -- (a3) -- (a4) -- (a1) -- (a5);
\draw[thick] (a1) -- (a3) -- (a6) -- (a2) -- (a4) -- (a6);
\draw[thick,-stealth] (a1) -- (infty);
\draw[fill] (a1) circle (1.75pt);
\draw[fill] (a2) circle (1.75pt);
\draw[fill] (a3) circle (1.75pt);
\draw[fill] (a4) circle (1.75pt);
\draw[fill] (a5) circle (1.75pt);
\draw[fill] (a6) circle (1.75pt);
\node at (infty) [anchor=south] {$\infty$};
\end{tikzpicture}
\qquad
\begin{tikzpicture}[scale=0.75]
\coordinate (a1) at (0,0);
\coordinate (a2) at (0,2);
\coordinate (a3) at (2,0);
\coordinate (a4) at (2,2);
\coordinate (a5) at (-2,0);
\coordinate (a6) at (3.4,3.4);
\coordinate (infty) at (6,0);
\draw[thick] (a1) -- (a2) -- (a3) -- (a4) -- (a1) -- (a5);
\draw[thick] (a1) -- (a3) -- (a6) -- (a2) -- (a4) -- (a6);
\draw[thick,-stealth] (a3) -- (infty);
\draw[fill] (a1) circle (1.75pt);
\draw[fill] (a2) circle (1.75pt);
\draw[fill] (a3) circle (1.75pt);
\draw[fill] (a4) circle (1.75pt);
\draw[fill] (a5) circle (1.75pt);
\draw[fill] (a6) circle (1.75pt);
\node at (infty) [anchor=south] {$\infty$};
\end{tikzpicture}
\\[5pt]
\caption{All non-isometric graphs obtained by attaching a lead to the right graph in Figure~\ref{fig.2}.}\label{fig.4}
\end{figure}

\begin{theorem}\label{thm:determ-6v}
Let $\Graph$ be a quantum graph on no more than 6 vertices satisfying the Assumption~\ref{assum:graph} and \ref{assum:potential-weak}.
Then the $S$-function and the eigenvalues (normal and embedded) of the Schrödinger equation on $\Graph_\infty$ uniquely determine the shape of $\Graph$.
\end{theorem}

\begin{proof}
Due to Lemma~\ref{lem:asympt} and Corollary~\ref{cor:asympt} knowing $S(\sqrt{\lambda})$ we can find
\begin{equation}
\label{4.2}
\hat{E}(\sqrt{\lambda})=\hat{\phi}_N(\lambda)+i\sqrt{\lambda}\hat{\phi}_D(\lambda).
\end{equation}
The  leading terms  $\hat{\phi}_N$ -- which, we recall, are the characteristic functions of Problem I for the free Laplacian -- are different for all the compact equilateral simple connected graphs with $V\leq 5$ vertices are different (see the formulae for $\phi_i$, $i=1,2,\ldots,30$ in  \cite{CP}). If we admit the number of vertices to be 6, then all the functions $\hat{\phi}_N$ are different except for the function given by the equation \eqref{4.1},
which corresponds to both graphs of Figure~\ref{fig.2} (see \cite{CP2}). If the leading terms $\hat{\phi}_N$ are different, then also the leading terms $\hat{E}$ of the Jost functions are different, $\hat{S}$ are different and the asymptotics of the $S$-function uniquely determine the shape of the graph. 

Both ways of attaching a lead to the left graph in Figure~\ref{fig.2} are shown at Figure~\ref{fig.3}, and using Theorem~\ref{thm:main}
we obtain the functions $\hat{\phi}_D$ corresponding to both graphs:
\begin{equation}
\begin{aligned}
\hat{\phi}_D(\lambda)&=(-512\cos^5\sqrt{\lambda}\ell+320\cos^3\sqrt{\lambda}\ell+96\cos^2\sqrt{\lambda}\ell-2\cos\sqrt{\lambda}\ell-2)\left(\frac{\sin\sqrt{\lambda}\ell}{\sqrt{\lambda}}\right)^5\\
\hat{\phi}_D(\lambda)&=(-256\cos^5\sqrt{\lambda}\ell+144\cos^3\sqrt{\lambda}\ell+24\cos^2\sqrt{\lambda}\ell-10\cos\sqrt{\lambda}\ell-2)\left(\frac{\sin\sqrt{\lambda}\ell}{\sqrt{\lambda}}\right)^5.
\end{aligned}
\end{equation}

Likewise, there are four ways of attaching a lead to the right graph in Figure~\ref{fig.2}, see Figure~\ref{fig.4}. The functions $\hat{\phi}_D$ corresponding to the graphs of Figure~\ref{fig.4} are given by the formulae:

\begin{equation}
\begin{aligned}
\hat{\phi}_D(\lambda)&=(-768\cos^5\sqrt{\lambda}\ell+480\cos^3\sqrt{\lambda}\ell+192\cos^2\sqrt{\lambda}\ell+21\cos\sqrt{\lambda}\ell)\left(\frac{\sin\sqrt{\lambda}\ell}{\sqrt{\lambda}}\right)^5\\
\hat{\phi}_D(\lambda)&=(-256\cos^5\sqrt{\lambda}\ell+160\cos^3\sqrt{\lambda}\ell+32\cos^2\sqrt{\lambda}\ell-9cos\sqrt{\lambda}\ell-2)\left(\frac{\sin\sqrt{\lambda}\ell}{\sqrt{\lambda}}\right)^5\\
\hat{\phi}_D(\lambda)&=(-192\cos^5\sqrt{\lambda}\ell+116\cos^3\sqrt{\lambda}\ell+14\cos^2\sqrt{\lambda}\ell-11\cos\sqrt{\lambda}\ell-2)\left(\frac{\sin\sqrt{\lambda}\ell}{\sqrt{\lambda}}\right)^5\\
\hat{\phi}_D(\lambda)&=(-192\cos^5\sqrt{\lambda}\ell-84\cos^3\sqrt{\lambda}\ell+30\cos^2\sqrt{\lambda}\ell+3cos\sqrt{\lambda}\ell)\left(\frac{\sin\sqrt{\lambda}\ell}{\sqrt{\lambda}}\right)^5.
\end{aligned}
\end{equation}
It is clear that among the functions $\hat{\phi}_D$ corresponding to the graphs of Figure~\ref{fig.3} and Figure~\ref{fig.4} there are no coinciding pairs: hence, the  leading  terms of the Jost functions $\hat{E}$ corresponding to all the graphs in Figure~\ref{fig.3} and~\ref{fig.4} are all different (neither of them can be expressed as a multiple of another one). Therefore, the $S$-matrix uniquely determines the shapes of these graphs. 
\end{proof}

\section{Inverse problem for  general small trees}\label{sec:inverse-tree}


It was proved in \cite{CP} that the functions $\hat{\phi}_N$ are all different for trees with on at most 8 vertices; and in \cite{Pist} that there is only one pair of co-spectral trees on 9 vertices, cf. Figure~\ref{fig.5}.

\begin{figure}[H]
\centering
\begin{tikzpicture}[scale=0.55]
\coordinate (a1) at (0,0);
\coordinate (a2) at (0,2);
\coordinate (a3) at (0,4);
\coordinate (a4) at (0,6);
\coordinate (a5) at (0,8);
\coordinate (a6) at (0,10);
\coordinate (a7) at (-2,8);
\coordinate (a8) at (2,8);
\coordinate (a9) at (2,6);
\draw[thick] (a1) -- (a6);
\draw[thick] (a7) -- (a8);
\draw[thick] (a4) -- (a9);
\draw[fill] (a1) circle (1.75pt);
\draw[fill] (a2) circle (1.75pt);
\draw[fill] (a3) circle (1.75pt);
\draw[fill] (a4) circle (1.75pt);
\draw[fill] (a5) circle (1.75pt);
\draw[fill] (a6) circle (1.75pt);
\draw[fill] (a7) circle (1.75pt);
\draw[fill] (a8) circle (1.75pt);
\draw[fill] (a9) circle (1.75pt);
\end{tikzpicture}
\qquad
\begin{tikzpicture}[scale=0.55]
\coordinate (a1) at (0,0);
\coordinate (a2) at (0,2);
\coordinate (a3) at (0,4);
\coordinate (a4) at (0,6);
\coordinate (a5) at (0,8);
\coordinate (a6) at (-2,4);
\coordinate (a7) at (-2,8);
\coordinate (a8) at (2,8);
\coordinate (a9) at (2,4);
\draw[thick] (a1) -- (a5);
\draw[thick] (a7) -- (a8);
\draw[thick] (a6) -- (a9);
\draw[fill] (a1) circle (1.75pt);
\draw[fill] (a2) circle (1.75pt);
\draw[fill] (a3) circle (1.75pt);
\draw[fill] (a4) circle (1.75pt);
\draw[fill] (a5) circle (1.75pt);
\draw[fill] (a6) circle (1.75pt);
\draw[fill] (a7) circle (1.75pt);
\draw[fill] (a8) circle (1.75pt);
\draw[fill] (a9) circle (1.75pt);
\end{tikzpicture}
\\[5pt]
\caption{Co-spectral trees of 9 vertices, taken from \cite{Pist}.}\label{fig.5}
\end{figure}
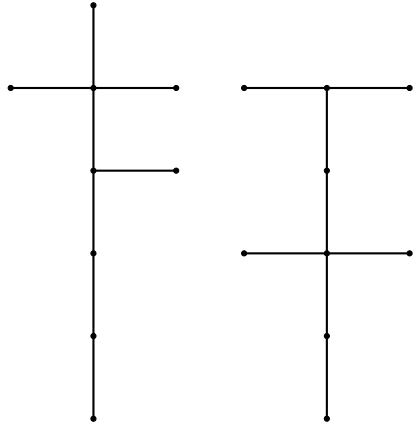


The corresponding characteristic polynomial is 
$$
P_{\Graph,\Graph}(z)=-4z^3(z^2-1)(4z^2-3)(3z^2-1).
$$
The corresponding  leading term  of  characteristic function is
\begin{equation}
\label{5.1}
\hat{\phi}_N(\lambda)=4\sqrt{\lambda}\sin\sqrt{\lambda}\ell\cos^3\sqrt{\lambda}\ell(4\cos^2\sqrt{\lambda}\ell-3)(3\cos^2\sqrt{\lambda}\ell-1),\qquad \lambda\in \C.
\end{equation}

All non-isometric graphs obtained by attaching one lead to the left tree in Figure~\ref{fig.5} are shown in Figure~\ref{fig.6} below, while all non-isometric graphs obtained by attaching a lead to the right tree in Figure~\ref{fig.5} are shown in Figure~\ref{fig.7} below.


\begin{theorem}\label{thm:tree}
 Let $\mathcal{T}_\infty$ be a metric tree obtained by attaching a lead to an equilateral compact metric tree $\mathcal T$ of not more than 9 vertices. Under the Assumptions~\ref{assum:potential-weak}
, let $q_0=0$.
Then the $S$-function and the eigenvalues (both normal and embedded into the continuous spectrum) of $\mathcal T_\infty$ uniquely determine the shape of $\mathcal{T}$.
\end{theorem}

\begin{proof}
Again, by Lemma~\ref{lem:asympt} and Corollary~\ref{cor:asympt} we know that $\hat{E}(\sqrt{\lambda})$ can be derived from $S(\sqrt{\lambda})$. The functions $\hat{\phi}_N$ are different for all the trees with $V\leq 8$ vertices (see the formulae for $\varphi_i$, $i=31,\ldots,70$, in the proof of \cite[Theorem~6.3]{CP}). If we admit the number of the vertices to be 9, then all the functions $\hat{\phi}_N$ are different, except for the function given by equation \eqref{5.1} which corresponds to the trees of Figure~\ref{fig.5} (see \cite{CP2}). If the functions $\hat{\phi}_D$ are different then also both the Jost functions $\hat{E}$
and the generalized $S$-functions $\hat{S}$
are different: we conclude that the the asymptotics of the $S$-function uniquely determine the shape of the tree. 


Likewise, there are seven ways of attaching a lead to the left tree in Figure~\ref{fig.5}, cf.~Figure~\ref{fig.6}.
\begin{figure}[H]
\centering
\begin{tikzpicture}[scale=0.49]
\coordinate (a1) at (0,0);
\coordinate (a2) at (0,2);
\coordinate (a3) at (0,4);
\coordinate (a4) at (0,6);
\coordinate (a5) at (0,8);
\coordinate (a6) at (0,10);
\coordinate (a7) at (-2,8);
\coordinate (a8) at (2,8);
\coordinate (a9) at (2,6);
\coordinate (infty6) at (4,10);
\node at (infty6) [anchor=south] {$\infty$};
\draw[thick,-stealth] (a6) -- (infty6);
\draw[thick] (a1) -- (a6);
\draw[thick] (a7) -- (a8);
\draw[thick] (a4) -- (a9);
\draw[fill] (a1) circle (1.75pt);
\draw[fill] (a2) circle (1.75pt);
\draw[fill] (a3) circle (1.75pt);
\draw[fill] (a4) circle (1.75pt);
\draw[fill] (a5) circle (1.75pt);
\draw[fill] (a6) circle (1.75pt);
\draw[fill] (a7) circle (1.75pt);
\draw[fill] (a8) circle (1.75pt);
\draw[fill] (a9) circle (1.75pt);
\end{tikzpicture}
\qquad
\begin{tikzpicture}[scale=0.49]
\coordinate (a1) at (0,0);
\coordinate (a2) at (0,2);
\coordinate (a3) at (0,4);
\coordinate (a4) at (0,6);
\coordinate (a5) at (0,8);
\coordinate (a6) at (0,10);
\coordinate (a7) at (-2,8);
\coordinate (a8) at (2,8);
\coordinate (a9) at (2,6);
\coordinate (infty4) at (6,6);
\node at (infty4) [anchor=south] {$\infty$};
\draw[thick,-stealth] (a4) -- (infty4);
\draw[thick] (a1) -- (a6);
\draw[thick] (a7) -- (a8);
\draw[fill] (a1) circle (1.75pt);
\draw[fill] (a2) circle (1.75pt);
\draw[fill] (a3) circle (1.75pt);
\draw[fill] (a4) circle (1.75pt);
\draw[fill] (a5) circle (1.75pt);
\draw[fill] (a6) circle (1.75pt);
\draw[fill] (a7) circle (1.75pt);
\draw[fill] (a8) circle (1.75pt);
\draw[fill] (a9) circle (1.75pt);
\end{tikzpicture}
\qquad
\begin{tikzpicture}[scale=0.49]
\coordinate (a1) at (0,0);
\coordinate (a2) at (0,2);
\coordinate (a3) at (0,4);
\coordinate (a4) at (0,6);
\coordinate (a5) at (0,8);
\coordinate (a6) at (0,10);
\coordinate (a7) at (-2,8);
\coordinate (a8) at (2,8);
\coordinate (a9) at (2,6);
\coordinate (infty5) at (-3.4,4.6);
\node at (infty5) [anchor=south] {$\infty$};
\draw[thick,-stealth] (a5) -- (infty5);
\draw[thick] (a1) -- (a6);
\draw[thick] (a7) -- (a8);
\draw[thick] (a4) -- (a9);
\draw[fill] (a1) circle (1.75pt);
\draw[fill] (a2) circle (1.75pt);
\draw[fill] (a3) circle (1.75pt);
\draw[fill] (a4) circle (1.75pt);
\draw[fill] (a5) circle (1.75pt);
\draw[fill] (a6) circle (1.75pt);
\draw[fill] (a7) circle (1.75pt);
\draw[fill] (a8) circle (1.75pt);
\draw[fill] (a9) circle (1.75pt);
\end{tikzpicture}
\qquad
\begin{tikzpicture}[scale=0.49]
\coordinate (a1) at (0,0);
\coordinate (a2) at (0,2);
\coordinate (a3) at (0,4);
\coordinate (a4) at (0,6);
\coordinate (a5) at (0,8);
\coordinate (a6) at (0,10);
\coordinate (a7) at (-2,8);
\coordinate (a8) at (2,8);
\coordinate (a9) at (2,6);
\coordinate (infty4) at (-4,6);
\node at (infty4) [anchor=south] {$\infty$};
\draw[thick,-stealth] (a4) -- (infty4);
\draw[thick] (a1) -- (a6);
\draw[thick] (a7) -- (a8);
\draw[thick] (a4) -- (a9);
\draw[fill] (a1) circle (1.75pt);
\draw[fill] (a2) circle (1.75pt);
\draw[fill] (a3) circle (1.75pt);
\draw[fill] (a4) circle (1.75pt);
\draw[fill] (a5) circle (1.75pt);
\draw[fill] (a6) circle (1.75pt);
\draw[fill] (a7) circle (1.75pt);
\draw[fill] (a8) circle (1.75pt);
\draw[fill] (a9) circle (1.75pt);
\end{tikzpicture}
\qquad
\begin{tikzpicture}[scale=0.49]
\coordinate (a1) at (0,0);
\coordinate (a2) at (0,2);
\coordinate (a3) at (0,4);
\coordinate (a4) at (0,6);
\coordinate (a5) at (0,8);
\coordinate (a6) at (0,10);
\coordinate (a7) at (-2,8);
\coordinate (a8) at (2,8);
\coordinate (a9) at (2,6);
\coordinate (infty3) at (4,4);
\node at (infty3) [anchor=south] {$\infty$};
\draw[thick,-stealth] (a3) -- (infty3);
\draw[thick] (a1) -- (a6);
\draw[thick] (a7) -- (a8);
\draw[thick] (a4) -- (a9);
\draw[fill] (a1) circle (1.75pt);
\draw[fill] (a2) circle (1.75pt);
\draw[fill] (a3) circle (1.75pt);
\draw[fill] (a4) circle (1.75pt);
\draw[fill] (a5) circle (1.75pt);
\draw[fill] (a6) circle (1.75pt);
\draw[fill] (a7) circle (1.75pt);
\draw[fill] (a8) circle (1.75pt);
\draw[fill] (a9) circle (1.75pt);
\end{tikzpicture}
\qquad
\begin{tikzpicture}[scale=0.49]
\coordinate (a1) at (0,0);
\coordinate (a2) at (0,2);
\coordinate (a3) at (0,4);
\coordinate (a4) at (0,6);
\coordinate (a5) at (0,8);
\coordinate (a6) at (0,10);
\coordinate (a7) at (-2,8);
\coordinate (a8) at (2,8);
\coordinate (a9) at (2,6);
\coordinate (infty2) at (4,2);
\node at (infty2) [anchor=south] {$\infty$};
\draw[thick,-stealth] (a2) -- (infty2);
\draw[thick] (a1) -- (a6);
\draw[thick] (a7) -- (a8);
\draw[thick] (a4) -- (a9);
\draw[fill] (a1) circle (1.75pt);
\draw[fill] (a2) circle (1.75pt);
\draw[fill] (a3) circle (1.75pt);
\draw[fill] (a4) circle (1.75pt);
\draw[fill] (a5) circle (1.75pt);
\draw[fill] (a6) circle (1.75pt);
\draw[fill] (a7) circle (1.75pt);
\draw[fill] (a8) circle (1.75pt);
\draw[fill] (a9) circle (1.75pt);
\end{tikzpicture}
\qquad
\begin{tikzpicture}[scale=0.49]
\coordinate (a1) at (0,0);
\coordinate (a2) at (0,2);
\coordinate (a3) at (0,4);
\coordinate (a4) at (0,6);
\coordinate (a5) at (0,8);
\coordinate (a6) at (0,10);
\coordinate (a7) at (-2,8);
\coordinate (a8) at (2,8);
\coordinate (a9) at (2,6);
\coordinate (infty1) at (4,0);
\node at (infty1) [anchor=south] {$\infty$};
\draw[thick,-stealth] (a1) -- (infty1);
\draw[thick] (a1) -- (a6);
\draw[thick] (a7) -- (a8);
\draw[thick] (a4) -- (a9);
\draw[fill] (a1) circle (1.75pt);
\draw[fill] (a2) circle (1.75pt);
\draw[fill] (a3) circle (1.75pt);
\draw[fill] (a4) circle (1.75pt);
\draw[fill] (a5) circle (1.75pt);
\draw[fill] (a6) circle (1.75pt);
\draw[fill] (a7) circle (1.75pt);
\draw[fill] (a8) circle (1.75pt);
\draw[fill] (a9) circle (1.75pt);
\end{tikzpicture}
\\[5pt]
\caption{Co-spectral trees obtained by attaching one lead to the left tree of Figure~\ref{fig.5}.}\label{fig.6}
\end{figure}
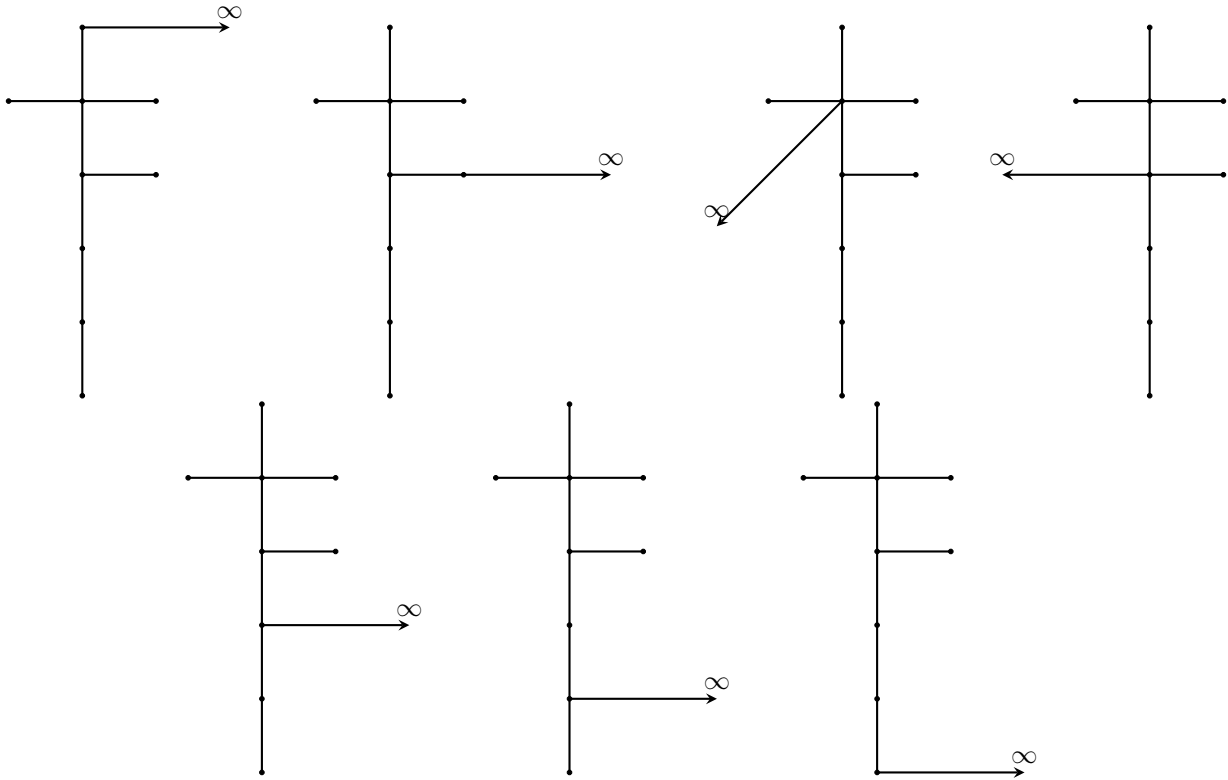


The functions $\hat{\phi}_D$ for all graphs in Figure~\ref{fig.6}  are the following ones, which we calculated using Theorem~\ref{thm:main}.


\begin{equation}
\begin{split}
\hat{\phi}_D(\lambda)&=48\cos^8\sqrt{\lambda}\ell-88\cos^6\sqrt{\lambda}\ell +49\cos^4\sqrt{\lambda}\ell-8\cos^2\sqrt{\lambda}\ell,\\
\hat{\phi}_D(\lambda)&=48\cos^8\sqrt{\lambda}\ell-84\cos^6\sqrt{\lambda}\ell +40\cos^4\sqrt{\lambda}\ell-3\cos^2\sqrt{\lambda}\ell,\\
\hat{\phi}_D(\lambda)&=12\cos^8\sqrt{\lambda}\ell-15\cos^6\sqrt{\lambda}\ell +4\cos^4\sqrt{\lambda}\ell,\\
\hat{\phi}_D(\lambda)&=16\cos^8\sqrt{\lambda}\ell-24\cos^6\sqrt{\lambda}\ell +9\cos^4\sqrt{\lambda}\ell,\\
\hat{\phi}_D(\lambda)&=24\cos^8\sqrt{\lambda}\ell-40\cos^6\sqrt{\lambda}\ell +20\cos^4\sqrt{\lambda}\ell-3\cos^2\sqrt{\lambda}\ell,\\
\hat{\phi}_D(\lambda)&=24\cos^8\sqrt{\lambda}\ell-35\cos^6\sqrt{\lambda}\ell +12\cos^4\sqrt{\lambda}\ell,\\
\hat{\phi}_D(\lambda)&=48\cos^8\sqrt{\lambda}\ell-76\cos^6\sqrt{\lambda}\ell +32\cos^4\sqrt{\lambda}\ell-3\cos^2\sqrt{\lambda}\ell.
\end{split}
\end{equation}

Likewise, 
there are seven ways of attaching a lead to the right tree in Figure~\ref{fig.5}, see Figure~\ref{fig.7}.

\begin{figure}[H]
\centering
\begin{tikzpicture}[scale=0.49]
\coordinate (a1) at (0,0);
\coordinate (a2) at (0,2);
\coordinate (a3) at (0,4);
\coordinate (a4) at (0,6);
\coordinate (a5) at (0,8);
\coordinate (a6) at (-2,4);
\coordinate (a7) at (-2,8);
\coordinate (a8) at (2,8);
\coordinate (a9) at (2,4);
\coordinate (infty8) at (6,8);
\node at (infty8) [anchor=south] {$\infty$};
\draw[thick,-stealth] (a8) -- (infty8);
\draw[thick] (a1) -- (a5);
\draw[thick] (a7) -- (a8);
\draw[thick] (a6) -- (a9);
\draw[fill] (a1) circle (1.75pt);
\draw[fill] (a2) circle (1.75pt);
\draw[fill] (a3) circle (1.75pt);
\draw[fill] (a4) circle (1.75pt);
\draw[fill] (a5) circle (1.75pt);
\draw[fill] (a6) circle (1.75pt);
\draw[fill] (a7) circle (1.75pt);
\draw[fill] (a8) circle (1.75pt);
\draw[fill] (a9) circle (1.75pt);
\end{tikzpicture}
\qquad
\begin{tikzpicture}[scale=0.49]
\coordinate (a1) at (0,0);
\coordinate (a2) at (0,2);
\coordinate (a3) at (0,4);
\coordinate (a4) at (0,6);
\coordinate (a5) at (0,8);
\coordinate (a6) at (-2,4);
\coordinate (a7) at (-2,8);
\coordinate (a8) at (2,8);
\coordinate (a9) at (2,4);
\coordinate (infty5) at (3.4,5.6);
\node at (infty5) [anchor=south] {$\infty$};
\draw[thick,-stealth] (a5) -- (infty5);
\draw[thick] (a1) -- (a5);
\draw[thick] (a7) -- (a8);
\draw[thick] (a6) -- (a9);
\draw[fill] (a1) circle (1.75pt);
\draw[fill] (a2) circle (1.75pt);
\draw[fill] (a3) circle (1.75pt);
\draw[fill] (a4) circle (1.75pt);
\draw[fill] (a5) circle (1.75pt);
\draw[fill] (a6) circle (1.75pt);
\draw[fill] (a7) circle (1.75pt);
\draw[fill] (a8) circle (1.75pt);
\draw[fill] (a9) circle (1.75pt);
\end{tikzpicture}
\qquad
\begin{tikzpicture}[scale=0.49]
\coordinate (a1) at (0,0);
\coordinate (a2) at (0,2);
\coordinate (a3) at (0,4);
\coordinate (a4) at (0,6);
\coordinate (a5) at (0,8);
\coordinate (a6) at (-2,4);
\coordinate (a7) at (-2,8);
\coordinate (a8) at (2,8);
\coordinate (a9) at (2,4);
\coordinate (infty4) at (4,6);
\node at (infty4) [anchor=south] {$\infty$};
\draw[thick,-stealth] (a4) -- (infty4);
\draw[thick] (a1) -- (a5);
\draw[thick] (a7) -- (a8);
\draw[thick] (a6) -- (a9);
\draw[fill] (a1) circle (1.75pt);
\draw[fill] (a2) circle (1.75pt);
\draw[fill] (a3) circle (1.75pt);
\draw[fill] (a4) circle (1.75pt);
\draw[fill] (a5) circle (1.75pt);
\draw[fill] (a6) circle (1.75pt);
\draw[fill] (a7) circle (1.75pt);
\draw[fill] (a8) circle (1.75pt);
\draw[fill] (a9) circle (1.75pt);
\end{tikzpicture}
\qquad
\begin{tikzpicture}[scale=0.49]
\coordinate (a1) at (0,0);
\coordinate (a2) at (0,2);
\coordinate (a3) at (0,4);
\coordinate (a4) at (0,6);
\coordinate (a5) at (0,8);
\coordinate (a6) at (-2,4);
\coordinate (a7) at (-2,8);
\coordinate (a8) at (2,8);
\coordinate (a9) at (2,4);
\coordinate (infty9) at (6,4);
\node at (infty9) [anchor=south] {$\infty$};
\draw[thick,-stealth] (a9) -- (infty9);
\draw[thick] (a1) -- (a5);
\draw[thick] (a7) -- (a8);
\draw[thick] (a6) -- (a9);
\draw[fill] (a1) circle (1.75pt);
\draw[fill] (a2) circle (1.75pt);
\draw[fill] (a3) circle (1.75pt);
\draw[fill] (a4) circle (1.75pt);
\draw[fill] (a5) circle (1.75pt);
\draw[fill] (a6) circle (1.75pt);
\draw[fill] (a7) circle (1.75pt);
\draw[fill] (a8) circle (1.75pt);
\draw[fill] (a9) circle (1.75pt);
\end{tikzpicture}
\qquad
\begin{tikzpicture}[scale=0.49]
\coordinate (a1) at (0,0);
\coordinate (a2) at (0,2);
\coordinate (a3) at (0,4);
\coordinate (a4) at (0,6);
\coordinate (a5) at (0,8);
\coordinate (a6) at (-2,4);
\coordinate (a7) at (-2,8);
\coordinate (a8) at (2,8);
\coordinate (a9) at (2,4);
\coordinate (infty3) at (3.4,1.6);
\node at (infty3) [anchor=south] {$\infty$};
\draw[thick,-stealth] (a3) -- (infty3);
\draw[thick] (a1) -- (a5);
\draw[thick] (a7) -- (a8);
\draw[thick] (a6) -- (a9);
\draw[fill] (a1) circle (1.75pt);
\draw[fill] (a2) circle (1.75pt);
\draw[fill] (a3) circle (1.75pt);
\draw[fill] (a4) circle (1.75pt);
\draw[fill] (a5) circle (1.75pt);
\draw[fill] (a6) circle (1.75pt);
\draw[fill] (a7) circle (1.75pt);
\draw[fill] (a8) circle (1.75pt);
\draw[fill] (a9) circle (1.75pt);
\end{tikzpicture}
\qquad
\begin{tikzpicture}[scale=0.49]
\coordinate (a1) at (0,0);
\coordinate (a2) at (0,2);
\coordinate (a3) at (0,4);
\coordinate (a4) at (0,6);
\coordinate (a5) at (0,8);
\coordinate (a6) at (-2,4);
\coordinate (a7) at (-2,8);
\coordinate (a8) at (2,8);
\coordinate (a9) at (2,4);
\coordinate (infty2) at (4,2);
\node at (infty2) [anchor=south] {$\infty$};
\draw[thick,-stealth] (a2) -- (infty2);
\draw[thick] (a1) -- (a5);
\draw[thick] (a7) -- (a8);
\draw[thick] (a6) -- (a9);
\draw[fill] (a1) circle (1.75pt);
\draw[fill] (a2) circle (1.75pt);
\draw[fill] (a3) circle (1.75pt);
\draw[fill] (a4) circle (1.75pt);
\draw[fill] (a5) circle (1.75pt);
\draw[fill] (a6) circle (1.75pt);
\draw[fill] (a7) circle (1.75pt);
\draw[fill] (a8) circle (1.75pt);
\draw[fill] (a9) circle (1.75pt);
\end{tikzpicture}
\qquad
\begin{tikzpicture}[scale=0.49]
\coordinate (a1) at (0,0);
\coordinate (a2) at (0,2);
\coordinate (a3) at (0,4);
\coordinate (a4) at (0,6);
\coordinate (a5) at (0,8);
\coordinate (a6) at (-2,4);
\coordinate (a7) at (-2,8);
\coordinate (a8) at (2,8);
\coordinate (a9) at (2,4);
\coordinate (infty1) at (4,0);
\node at (infty1) [anchor=south] {$\infty$};
\draw[thick,-stealth] (a1) -- (infty1);
\draw[thick] (a1) -- (a5);
\draw[thick] (a7) -- (a8);
\draw[thick] (a6) -- (a9);
\draw[fill] (a1) circle (1.75pt);
\draw[fill] (a2) circle (1.75pt);
\draw[fill] (a3) circle (1.75pt);
\draw[fill] (a4) circle (1.75pt);
\draw[fill] (a5) circle (1.75pt);
\draw[fill] (a6) circle (1.75pt);
\draw[fill] (a7) circle (1.75pt);
\draw[fill] (a8) circle (1.75pt);
\draw[fill] (a9) circle (1.75pt);
\end{tikzpicture}
\\[5pt]
\caption{Co-spectral trees obtained by attaching one lead to the right tree of Figure~\ref{fig.5}.}\label{fig.7}
\end{figure}
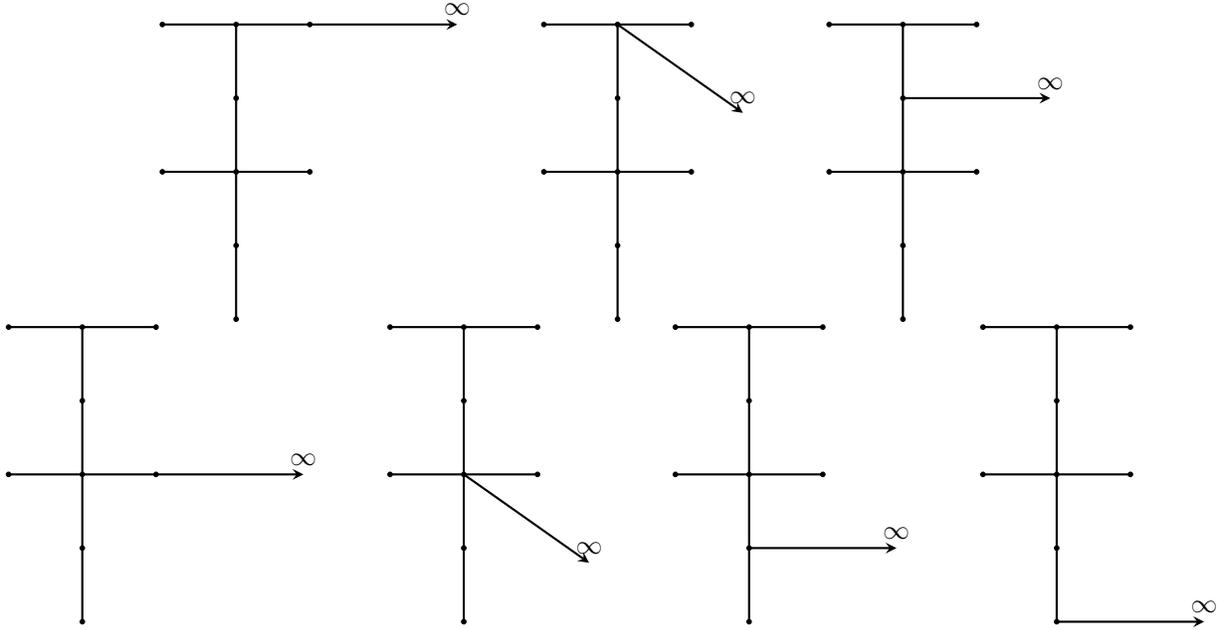

The functions $\hat{\phi}_D$ corresponding to the graphs in Figure~\ref{fig.7} are given by the following expressions:

\begin{equation}
\begin{split}
\hat{\phi}_D(\lambda)&=48\cos^8\sqrt{\lambda}\ell-76\cos^6\sqrt{\lambda}\ell+19\cos^4\sqrt{\lambda}\ell, \\
\hat{\phi}_D(\lambda)&=16\cos^8\sqrt{\lambda}\ell-20\cos^6\sqrt{\lambda}\ell+5\cos^4\sqrt{\lambda}\ell, \\
\hat{\phi}_D(\lambda)&=24\cos^8\sqrt{\lambda}\ell-13\cos^6\sqrt{\lambda}\ell-8\cos^4\sqrt{\lambda}\ell +4\cos^2\sqrt{\lambda}\ell, \\
\hat{\phi}_D(\lambda)&=48\cos^8\sqrt{\lambda}\ell-88\cos^6\sqrt{\lambda}\ell+48\cos^4\sqrt{\lambda}\ell-7\cos^2\sqrt{\lambda}\ell, \\
\hat{\phi}_D(\lambda)&=12\cos^8\sqrt{\lambda}\ell-16\cos^6\sqrt{\lambda}\ell+5\cos^4\sqrt{\lambda}\ell, \\
\hat{\phi}_D(\lambda)&=24\cos^8\sqrt{\lambda}\ell-32\cos^6\sqrt{\lambda}\ell+9\cos^4\sqrt{\lambda}\ell, \\
\hat{\phi}_D(\lambda)&=48\cos^8\sqrt{\lambda}\ell-84\cos^6\sqrt{\lambda}\ell+44\cos^4\sqrt{\lambda}\ell- 7\cos^2\sqrt{\lambda}\ell.
\end{split}
\end{equation}
This concludes the proof.
\end{proof}

\section{Inverse problem for small fuzzy balls}
\label{sec:exa-butgrout}

Let us now recall a class of simple, non-bipartite graphs identified in \cite{ButGro11}. So-called \textit{fuzzy balls} $\mFB_{r,s}$ are combinatorial graphs on $r+s+2$ vertices
constructed by attaching to a complete graph $\mK_n$ (the ``bulk'') on $n=r+s$ vertices $\mvdi,\ldots,\mv_n$ two vertices $\mw_1,\mw_2$ in such a way that $\mw_1$ is adjacent to $\mvdi,\ldots,\mv_r$ and $\mw_2$ is adjacent to $\mv_{r+1},\ldots,\mv_{r+s}$:  we call $\mw_1,\mw_2$ the \textit{off-bulk vertices}. Then Butler and Grout showed in \cite[Example~1]{ButGro11} that, given $n\in \N$, $n\ge 4$, 
the class
\[
{\mathcal{FB}_n}:=\{\mFB_{r,s}:r+s=n\}
\]
consists of combinatorial graphs with same number of vertices and edges, and that additionally are co-spectral with respect to the normalized Laplacian.  Observe that we have already encountered fuzzy balls: indeed, the graphs in Figure~\ref{fig.2} are precisely $\mFB_{1,3}$ and $\mFB_{2,2}$.

It is known \cite{ButGro11}  that, for any given $n$, in case of $q_{\me}\equiv 0$ on all edges (free Laplacians) and standard vertex conditions all fuzzy balls in $\mathcal{FB}_n$ are co-spectral. That means that they have the same number of edges, and are co-spectral with respect to the normalized Laplacians.
Thus, we cannot distinguish them using the first and the second terms of the eigenvalue asymptotics in case of the Schrödinger problem. 
 Let us show that we can, however, resolve the co-spectrality by attaching one lead. 
 \begin{theorem} Let $\mathsf{G}_i\in{\mathcal{FB}_n}$ ($i=1,2$) with $r_i\geq 1$, $s_i\geq 1$, $r_i+s_i=n\geq 4$, $r_1\not=r_2$, and consider the corresponding equilateral graphs $\Graph_i$, each of whose edges has length $\ell$. 
Under the Assumption~\ref{assum:potential-weak}, 
attach one lead to either of the off-bulk vertices of $\mG_1$, and one lead to either of the off-bulk vertices of $\mG_2$. Let $q_0=0$ on each of attached leads.

Then we obtain different $\hat{\phi}_D$ and can thus distinguish the co-spectral metric graphs $\Graph_1,\Graph_2$. 
\end{theorem}


 \begin{proof}
For each $\mG\in{\mathcal {FB}}_n$, the total number of the vertices is $V=n+2$ and the total number of the edges is $E=\frac{n(n-1)}{2}+n=\frac{n(n+1)}{2}$. Then $E-V=\frac{n^2-n-4}{2}$: Upon associating the metric graphs $\Graph_1$ and $\Graph_2$ with any two combinatorial graphs $\mG_1,\mG_2\in{\mathcal{FB}_n}$, by Theorem~\ref{thm:main}
 we have
$$
\hat{\phi}_{D,1}(\lambda)=\left(\frac{\sin\sqrt{\lambda}\ell}{\sqrt{\lambda}}\right)^{ \frac{n^2-n-4}{2}+r_1} P_{\Graph,\tilde{\Graph}_1}(\cos\sqrt{\lambda}\ell)
$$
for the first graph and
$$
\hat{\phi}_{D,2}(\lambda)=\left(\frac{\sin\sqrt{\lambda}\ell}{\sqrt{\lambda}}\right)^{ \frac{n^2-n-4}{2}+r_2} P_{\Graph,\tilde{\Graph}_2}(\cos\sqrt{\lambda}\ell)
$$
for the second one.

Let $r_2>r_1$. Suppose 
$$\left(\frac{\sin\sqrt{\lambda}l}{\sqrt{\lambda}}\right)^{ \frac{n^2-n-4}{2}+r_2} P_{\Graph,\tilde{\Graph}_2}(\cos\sqrt{\lambda}\ell)\equiv\left(\frac{\sin\sqrt{\lambda}l}{\sqrt{\lambda}}\right)^{ \frac{n^2-n-4}{2}+r_1} P_{\Graph,\tilde{\Graph}_1}(\cos\sqrt{\lambda}\ell),
$$ 
that is,
\begin{equation}\label{eq:vyach-details}
\left(\frac{\sin\sqrt{\lambda}l}{\sqrt{\lambda}}\right)^{r_2-r_1} P_{\Graph,\tilde{\Graph}_2}(\cos\sqrt{\lambda}\ell)\equiv  P_{\Graph,\tilde{\Graph}_1}(\cos\sqrt{\lambda}\ell).
\end{equation}
Let $\{\alpha_i^{m}\}$ be the set of zeroes of $ P_{\Graph,\tilde{\Graph}_m}(z)$, $m=1,2$. Let $\theta$ be a real number and $\theta\notin\{\alpha_i^1\}\cup\{\alpha_i^2\}\cup\{0\}\cup\{\pi\}$. Let $\sqrt{\tau}:=\frac{1}{l}\arccos \theta$ and $\lambda_k=(\sqrt{\tau}+2\pi k)$. Then as $k\to +\infty$ the left hand side of \eqref{eq:vyach-details} tend to $0$, while the right hand side does not -- a contradiction.
%
\end{proof}


\begin{remark}
 If we attach one lead to any vertex of degree $n$, then we have the factor 
$$
\left(\frac{\sin\sqrt{\lambda}\ell}{\sqrt{\lambda}}\right)^{ \frac{n^2-n-4}{2}+n} =\left(\frac{\sin\sqrt{\lambda}\ell}{\sqrt{\lambda}}\right)^{ \frac{n^2+n-4}{2}},
$$
which is independent of $r$. Thus, we cannot state that it is possible to distinguish the co-spectral graphs $\Graph_1,\Graph_2$. 
\end{remark}


\begin{remark}
Fuzzy balls are, by construction, simple graphs on at least six vertices. Another class of graphs that are co-spectral with respect to the normalized Laplacian, again introduced in~\cite{ButGro11}, are so-called \emph{inflated stars}. In their easiest form, inflated stars $\mIS_{m_1,\ldots,m_k}$ are, for all $k\in \N$, combinatorial graphs on $V=m+n+k+1$ vertices and $E=2(m_1+\ldots+m_k)$ edges consisting of a central vertex $\mv_0$, $k$ external vertices $\mvdi,\ldots,\mv_k$, and, for all $i\in\{1,\ldots,k\}$, $2m_i$ parallel edges linking $\mv_0$ with $\mv_i$ -- each such edge subdivided into two  to make the graph (for technical reasons) bipartite. The smallest pair of co-spectral inflated stars consists of $\mIS_{1,3}$ and $\mIS_{2,2}$, which we show in Figure~\ref{fig.infl-star}.


\begin{figure}[H]
\begin{tikzpicture}
\coordinate (a) at (0,0);
\coordinate (a1) at (1,.8);
\coordinate (a2) at (1,-.8);
\coordinate (a3) at (1,0);
\coordinate (b) at (2,0);
\coordinate (b2) at (3,0);
\coordinate (c) at (4,0);
\draw[bend left=63]  (a) edge (a1);
\draw[bend right=63]  (a) edge (a2);
\draw[bend left=63]  (a1) edge (b);
\draw[bend right=63]  (a2) edge (b);
\draw  (a) edge (c);
\node at (a) [anchor=east] {$\mv_1$};
\node at (b) [anchor=south west] {$\mv_2$};
\node at (c) [anchor=west] {$\mv_3$};
\draw[fill] (a) circle (2pt);
\draw[fill] (b) circle (2pt);
\draw[fill] (a1) circle (2pt);
\draw[fill] (a2) circle (2pt);
\draw[fill] (a3) circle (2pt);
\draw[fill] (b2) circle (2pt);
\draw[fill] (c) circle (2pt);
\end{tikzpicture}
\quad
\begin{tikzpicture}
\coordinate (a) at (0,0);
\coordinate (a1) at (1,.8);
\coordinate (a2) at (1,-.8);
\coordinate (b) at (2,0);
\coordinate (b1) at (3,.8);
\coordinate (b2) at (3,-.8);
\coordinate (c) at (4,0);
\draw[bend left=63]  (a) edge (a1);
\draw[bend right=63]  (a) edge (a2);
\draw[bend left=63]  (a1) edge (b);
\draw[bend right=63]  (a2) edge (b);
\draw[bend left=63]  (b) edge (b1);
\draw[bend right=63]  (b) edge (b2);
\draw[bend left=63]  (b1) edge (c);
\draw[bend right=63]  (b2) edge (c);
\node at (a) [anchor=east] {$\mv_1$};
\node at (b) [anchor=west] {$\mv_2$};
\node at (c) [anchor=west] {$\mv_3$};
\draw[fill] (a) circle (2pt);
\draw[fill] (a1) circle (2pt);
\draw[fill] (a2) circle (2pt);
\draw[fill] (b1) circle (2pt);
\draw[fill] (b2) circle (2pt);
\draw[fill] (b) circle (2pt);
\draw[fill] (c) circle (2pt);
\end{tikzpicture}
\caption{The inflated stars $\mIS_{1,3}$ (left) and $\mIS_{2,2}$ (right).}\label{fig.infl-star}
\end{figure}
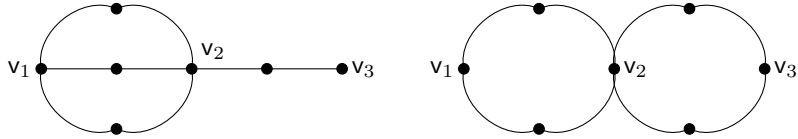



Inflated stars easily allow us to explain why we have imposed simplicity of the metric graph throughout Section~\ref{sec:inverse-tree}. Indeed, it is well-known that inserting or removing so-called \emph{dummy vertices} -- i.e., vertices of degree two -- has no impact on the spectrum of standard quantum graphs, see~\cite[Remark~2.1]{BerKenKur19}. Accordingly, the metric graphs whose underlying combinatorial graphs are $\mIS_{1,3},\mIS_{2,2}$ can be equivalently drawn as in Figure~\ref{fig.infl-star-metric}: these are co-spectral quantum graphs on 3 vertices. Though, they are no counter-example to the assertion, mentioned in the introduction, that the spectrum of the free Laplacian with standard vertex conditions on a \emph{simple}, connected graph on at most 6 vertices uniquely determines the shape of the graph.


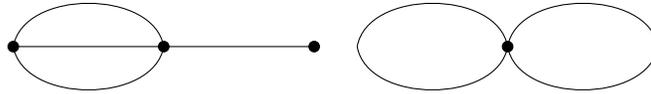
\begin{figure}[H]
\begin{tikzpicture}
\coordinate (a) at (0,0);
\coordinate (b) at (2,0);
\coordinate (c) at (4,0);
\draw[bend left=80]  (a) edge (b);
\draw[bend right=80]  (a) edge (b);
\draw  (a) edge (c);
\draw[fill] (a) circle (2pt);
\draw[fill] (b) circle (2pt);
\draw[fill] (c) circle (2pt);
\end{tikzpicture}
\quad
\begin{tikzpicture}
\coordinate (a) at (0,0);
\coordinate (b) at (2,0);
\coordinate (c) at (4,0);
\draw[bend left=80]  (a) edge (b);
\draw[bend right=80]  (a) edge (b);
\draw[bend left=80]  (b) edge (c);
\draw[bend right=80]  (b) edge (c);
\draw[fill] (b) circle (2pt);
\end{tikzpicture}
\caption{The non-simple metric graph corresponding to the inflated stars $\mIS_{1,3}$ (left) and $\mIS_{2,2}$  (right).}\label{fig.infl-star-metric}
\end{figure}


Incidentally, inflated stars also allow us to contribute to extremal spectral geometry of quantum graphs. In \cite[Theorem~4.2]{KenKurMal16} it was stated that the spectral gap for the free Laplacian with standard vertex conditions on metric graphs with $E$ edges is maximal for equilateral flowers and pumpkin graphs. It has been already pointed out in~\cite{BanLev17} that, for $E=2$, this assertion is false because \emph{all} flower graphs (i.e., figure-eight graphs) have same spectral gap.
 Now, inflated stars provide a class of (equilateral!) metric graphs on arbitrarily large number of edges such that the corresponding standard quantum graphs are co-spectral.
\end{remark}

\bibliographystyle{plain}

\begin{thebibliography}{1}

\bibitem{Ak} A.M. Akhtyamov, I.A. Trooshin. Direct and boundary inverse spectral problems for Sturm--Liouville differential operators on noncompact star-shaped graphs. Azerbaijan J. Math. (2019) 9:108--124.

\bibitem{BanLev17}
R.~Band and G.~L\'evy.
\newblock Quantum graphs which optimize the spectral gap.
\newblock {\em Ann.\ Henri Poincar\'e} (2017) 18:3269--3323.




\bibitem{Bar} V. Bargman. Remarks on the determination of the central field of force from the elastic scattering phase shifts. Phys. Rev. (1949) 75:301--314. 

\bibitem{Bel85}
{J.\ von} Below.
\newblock A characteristic equation associated with an eigenvalue problem on
 $c^2$-networks.
\newblock {\em Lin.\ Algebra Appl.} (1985) 71:309--325.

\bibitem{Bel01}
{J.\ von} Below.
\newblock Can one hear the shape of a network?
\newblock In F.\ {Ali Mehmeti}, {J.\ von} Below, and S.\ Nicaise, editors, {\em
 Partial Differential Equations on Multistructures (Proc.\ Luminy 1999)},
 volume 219 of {\em Lect.\ Notes Pure Appl.\ Math.}, pages 19--36, New York,
 2001. Marcel Dekker.

\bibitem{BerKenKur19}
G.~Berkolaiko, J.B. Kennedy, P.~Kurasov, and D.~Mugnolo.
\newblock Surgery principles for the spectral analysis of quantum graphs.
\newblock {\em Trans.\ Amer.\ Math.\ Soc.} (2019) 372:5153--5197.

\bibitem{BK} G. Berkolaiko, P. Kuchment. Introduction to Quantum Graphs. Mathematical Surveys and Monographs, Vol.186, AMS, Providence RI, 2013.
%
\bibitem{BKS} J. Boman, P. Kurasov, R. Suhr. Schrödinger operators on graphs and geometry II. Spectral estimates for $L_1$-potentials and Ambartsumian's theorem. Integr. Eq. Oper. Theory (2018) 90:40
%
%

\bibitem{ButGro11}
S.~Butler and J.~Grout.
\newblock A construction of cospectral graphs for the normalized {L}aplacian.
\newblock {\em Electronic J.\ Combin.} (2011) P231--P231.

\bibitem{CveDooSac79}
D.M.\ Cvetkovi\'c, M.\ Doob, and H.\ Sachs.
\newblock {\em {Spectra of Graphs -- Theory and Applications}}.
\newblock Pure Appl.\ Math. Academic Press, New York, 1979.



%
%
%
\bibitem{CP} A. Chernyshenko, V. Pivovarchik. Recovering the shape of a quantum graph. Integr. Equ. Oper. Theory  (2020) 92:23. 
%
\bibitem{CP2} A. Chernyshenko, V. Pivovarchik. Cospectral quantum graphs. 2022 arXiv:2112.14235
%
%
%
%
\bibitem{EE} D.E. Edmunds, W.D. Evans. Spectral theory and differential operators. Clarendon Press, Oxford, 1989.
%
%
%

\bibitem{Gla} I.M. Glazman. Direct methods of qualitative spectral analysis of singular differential operators. Israeli Program for Scientific Translations, Jerusalem, 1965.	

\bibitem{GK} I. Gohberg, M. Krein. Introduction to the Theory of Linear Non-Selfadjoint Operators in Hilbert Space. AMS, 1969.
%
\bibitem{GS} B. Gutkin, U. Smilansky, Can one hear the shape of a graph? J. Phys. A Math. Gen. (2001), 34:6061--6068. 
%
%
\bibitem{KenKurMal16}
J.B. Kennedy, P.~Kurasov, G.~Malenová, and D.~Mugnolo.
\newblock On the spectral gap of a quantum graph.
\newblock {\em Ann.\ Henri Poincar\'e} (2016) 17:2439--2473.

\bibitem{KurSte02}
P.\ Kurasov and F.~Stenberg.
\newblock {On the inverse scattering problem on branching graphs}.
\newblock {\em J.\ Phys.\ A} (2002) 35:101--121.


\bibitem{LP} Y. Latushkin, V. Pivovarchik. Scattering in a forked-shaped waveguide. Integr. Equ. Oper. Theory (2008) 61:365--399. 
%
%
\bibitem{LawP} C.-K. Law, V. Pivovarchik. Characteristic functions of quantum graphs. J. Phys A: Math. Theor. (2009) 42:035302.
%



\bibitem{Lev} B.M. Levitan. Inverse Sturm--Liouville problems, VSP. Zeist, 1987. 

\bibitem{LleFabPos22}
F.~Lled{\'o}, J.S. Fabila-Carrasco, and O.~Post.
\newblock Isospectral graphs via spectral bracketing.
\newblock arXiv:2207.03924, 2022.
%


\bibitem{Mar} V.A. Marchenko. Sturm--Liouville Operators and Applications, revised edition. AMS Chelsea Publishing, Providence, RI 2011. 
%
\bibitem{KN} P.~Kurasov and S.~Naboko. Rayleigh estimates for differential operators on graphs. J. Spectr. Theory (2014) 4:211–219.
%
\bibitem{MP2} M. Möller, V. Pivovarchik, Direct and Inverse Finite-Dimensional Spectral Problems on Graphs.
Operator Theory: Advances and Applications, Vol.~283. Birkhäuser, Cham, 2020.
%
%
%
%

\bibitem{MP} M. M\" oller, V. Pivovarchik, Spectral Theory of Operator Pencils, Hermite--Biehler Functions, and Their Applications. Birkhäuser, Cham, 2015.
%
%
\bibitem{Mug19}
D.~Mugnolo.
\newblock What is actually a metric graph?
\newblock arXiv:1912.07549.

\bibitem{Nai} N.A. Naimark. Linear differential operators. Part II, Linear differential operators in Hilbert space. Frederic Ungar Publishing Co. NY, 1968. 


\bibitem{Nus} H.M. Nussenzweig. Causality and Dispersion Relations. Acad. Press, N.Y. 1972.



\bibitem{P1} V Pivovarchik. Scattering in a loop-shaped waveguide. In: Recent Advances in Operator Theory, Groningen, 1998. Birkhäuser, Basel and Boston, 2001: 527--543. 
%
\bibitem{Pist} M.-E.~Pistol. Generating isospectral but not isomorphic quantum graphs. arXiv: 2104.12885.
%
%
\bibitem{PW} V. Pivovarchik, H. Woracek. Sums of Nevanlinna functions and differential equations on a star-shaped graphs. Oper. Matrices, (2009) 3:451--501. 


\bibitem{Pokor} Yu. Pokorny, O. Penkin, V. Pryadiev, A. Borovskih, K. Lazarev, S. Shabrov. Differential equations on geometric graphs (in Russian), Fizmatlit, 2005. 
%
\bibitem{Re} T. Regge. Construction of potential from resonances. Nuovo
Cimento  (1958) 9:491--503 and 671--679.

\bibitem{Rich} R.D. Richtmyer. Principles of Advanced Mathematical Physics. Springer, NY, 1978.
%

\bibitem{Rot84}
J.-P.\ Roth.
\newblock Le spectre du laplacien sur un graphe.
\newblock In G.\ Mokobodzki and D.\ Pinchon, editors, {\em Colloque de
 Th{\'e}orie du Potentiel - Jacques Deny (Proc.\ Orsay 1983)}, volume 1096 of
 {\em Lect.\ Notes.\ Math.}, pages 521--539, Berlin, 1984. Springer-Verlag.

%

%
%
%
%
%
%
%
%
%


%
%
%
%
%

\end{thebibliography}
%

\end{document}